\documentclass[a4paper,11pt, twoside]{article}
\usepackage{amsmath,amsthm}
\usepackage{amssymb,latexsym}
\usepackage{mathrsfs}
\usepackage{enumerate}
\usepackage[colorlinks,citecolor=blue]{hyperref}
\usepackage[numbers,sort&compress]{natbib}
\usepackage{indentfirst}

\DeclareMathAlphabet{\mathpzc}{OT1}{pzc}{m}{it}

\newtheorem{theorem}{Theorem}[section]
\newtheorem{corollary}{Corollary}[section]
\newtheorem{lemma}{Lemma}[section]

\newtheorem{remark}{Remark}[section]

\theoremstyle{definition} \theoremstyle{remark}
\numberwithin{equation}{section}
\allowdisplaybreaks

\date{}

\begin{document}
	
	
	\date{}
	\baselineskip 0.22in
	\title{{\bf Long-time dynamics for time-nonlocal generalized Rayleigh-Stokes equations}}
	
\author{Li Peng$^1$, Lin Deng$^1$, Jia Wei He$^{2,}\thanks{E-mail address:jwhe@gxu.edu.cn}$
		\\[1.8mm]
		\footnotesize  {$^1$Faculty of Mathematics and Computational Science, Xiangtan University, Hunan 411105,  China}\\
		\footnotesize  {$^2$School of Mathematics, Guangxi University, Nanning 530004, China}
	}

	\maketitle
	
	\begin{abstract}
	In this paper, we consider an autonomous semi-dynamical system driven by semilinear time-nonlocal evolution equations, these type equations are used to describe the Rayleigh-Stokes problem for a non-Newtonian fluid to a generalized second grade fluid. We first investigate the global well-posedness of solutions consisting of global Lipschitz condition by a weighted space $\mathcal C$. Utilizing the topology of uniform convergence on compact subsets of $\mathcal C$, we construct a semi-dynamical system that satisfies the semigroup structure. It also is shown that this semi-dynamical system has an attracting set when the vector field function satisfies a dissipativity condition and a local Lipschitz condition. With the asymptotic compactness, we also establish the existence of generalized attractors in $\mathcal C_\alpha$ a subspace of $\mathcal C$ with the weighted norm. \\[2mm]
	\end{abstract}
	
	\baselineskip 0.25in
	
	\section{Introduction}

It is well known that many significant complex fluids are non-Newtonian and exhibit time-dependent behaviors such as thixotropy and rheopexy \cite{Mahmood,Pandey}. The behavior of non-Newtonian fluids often follows a power-law relationship.
In particular, for a more generalized second-grade fluid, constructing a simple mathematical model that captures the diverse behaviors of non-Newtonian fluids becomes even more challenging. From this physical perspective, the constitutive relationship employed for the second-grade fluid takes the following form:
$$
{\bf T}=-p{\bf I} +\mu {\bf K}_1+\varpi_1{\bf K}_2+\varpi_2{\bf K}_1^2,
$$
where ${\bf I}$ is the identity tensor, ${\bf T}$ is the Cauchy stress tensor, $p$ is the hydrostatic pressure, $\mu\geq0$, $\varpi_1$ and $\varpi_2$ are normal stress
moduli.  ${\bf K}_1,{\bf K}_2$ are the kinematical tensors as
$$
{\bf K}_1=\nabla {\bf V}+(\nabla {\bf V})^{\rm T},\ \ {\bf K}_2=\partial_t {\bf K}_1+{\bf K}_1(\nabla {\bf V})+(\nabla {\bf V})^{\rm T}{\bf K}_1,
$$
where ${\bf V}$ is the velocity, $\nabla$ is the gradient operator and the
superscript $\rm T$ denotes a transpose operation, $\partial_t$ denotes the material time derivative.

For generalized second grade fluids, the constitutive equation of $\bf T$ still fit, however ${\bf K}_2$ should be defined as the form, (see e.g. \cite{Shen2006})
$$
{\bf K}_2=\partial_t^{*,k} {\bf K}_1+{\bf K}_1(\nabla {\bf V})+(\nabla {\bf V})^{\rm T}{\bf K}_1,
$$
where $\partial_t^{*,k}$ is the time-nonlocal derivative of type
	$$\partial_t^{*,k}v(t)=\frac{ {d}}{ {d}t}\int_0^t k(t-s)v(s) {d}s,~~~t>0,$$
	with nonnegative kernel $k\in L^1_{\mathrm{loc}}(\mathbb{R}^+)$.
Consider a special case of Rayleigh-Stokes problem for a heated edge, occupies the space the first quadrant of a rectangular edge $(x\geq0, -\infty<y<+\infty,z\geq0)$. 
The velocity field will be of the form ${\bf V}=u(x,z,t){\bf j},$
where $u$ is the velocity in the $y$ coordinate direction and ${\bf j}$ is the unit vector along $y$-coordinate direction. Note that, in the absence of body forces, the equation governing motion takes the form,
$\rho{D{\bf  V}}/{Dt}=\nabla\cdot {\bf T},$ in which $\rho$ is density of the fluid, $D/Dt$ is the material derivative, and the incompressible fluid $\nabla\cdot{\bf V}=0$ implying that
$$
\partial_t u(x,z,t) =(\nu+ \varrho \partial_t^{*,k})( \partial_{xx}+\partial_{zz} ) u(x,z,t),
$$
where $\nu=\mu/\rho$, $\varrho=\varpi_1/\rho$, $\partial_{xx}=\partial^2/\partial x^2$, $\partial_{zz}=\partial^2/\partial z^2$.
	
Inspired by this model, in this paper, we consider the long-time dynamic behaviors for the semilinear generalized Rayleigh-Stokes problem in time nonlocal equations of type
	\begin{equation}\label{eq:1.1}
		\left\{ \begin{aligned}& \partial_tu-(1+\partial_t^{*,k})\Delta u=f(u),\quad {\rm for}\ t>0,\\
			&u(0)=u_0,  \end{aligned}\right.
	\end{equation}
	where $\Delta$ is the Laplacian operator on $L^p(\mathbb R^d)$, $\partial_t= {\partial}/{\partial t}$, $f$ is a given function to be specified later.
	The proposed system is a general model for some problems studied in literature. Indeed, when $k$ is a constant, (\ref{eq:1.1}) is a classical diffusion equation, the related dynamic behaviors has been  studied in many works. If $k$ is a regular function, like $k\in C^1(\mathbb{R}^{+})$, then one gets
	$$\partial_tu-(1+k_0)\Delta u-\int_0^tk_1(t-s)\Delta u(s)\mathrm{d}s=f,$$
	with $k_0=k(0)$ and $k_1(t)=k^{\prime}(t)$, which is clearly a nonclassical diffusion equation with memory. Let $k(t)= k_0t^{-\alpha}/\Gamma(\alpha)$ for $k_0>0$, we can see that (\ref{eq:1.1}) is a fractional Rayleigh-Stokes equation, i.e., $	\partial_tu-(1+k_0\partial_t^\alpha)\Delta u=f.$ Here, the fractional derivative $\partial_t^\alpha$ serves to characterize the memory properties of the material.
	
	As an important representative of time-nonlocal evolution equations, the Rayleigh-Stokes equation occupies a fundamental position in the mechanics of viscoelastic fluids. It was initially proposed to describe the dynamic behavior of second-grade fluids \cite{Fetecau2009, Shen2006}, and is expressed as fractional Rayleigh-Stokes equation with $k_0=1$. In recent years, remarkable progress has been achieved in both the theoretical and numerical research on this equation, in terms of numerical computation, researchers have formulated a variety of numerical schemes to solve the Rayleigh-Stokes equation \cite{Bazhlekova2015, Bi2018, Chen2013, Chen2008}. In the study of analytical theory, analytical representations of solutions have been derived for the linear case \cite{Shen2006, Khan2009}. Moreover, recent studies have delved  into the regularity of solutions for nonlinear Rayleigh-Stokes equations \cite{He2022,Lan2022, Luc2021, Zhou2021}. Additionally, inverse problems associated with this class of equations have also attracted considerable attention \cite{Luc2019, Ngoc2021, Tuan2019}. More recently, Peng et al. \cite{Peng2026} studied the well-posedness results for a class of logarithmic nonlinearity of problem (\ref{eq:1.1}) with $f(u)\sim u|u|^{\gamma-2}\log|u|$, $\gamma>2$.
	Of more direct relevance to this paper are the well-posedness theory and the dynamic system behavior of time nonlocal differential equations. In the finite-dimensional context, the existence and uniqueness of solutions for fractional differential equations have been investigated \cite{Cong2020, Podlubny1998,Zhou} for example. For results on nonlocal evolution equations in Hilbert spaces, see e.g. \cite{Ke2020}.
	It is noteworthy that in practical applications, the parameters of the fractional equation and the initial data frequently hinge on material constants, and the determination of these constants is usually only moderately accurate \cite{Diethelm1999}. This naturally gives rise to a key issue: the continuous dependence of solutions on initial conditions and the nonlocal kernel, that is, the stability and structural stability of solutions. Such stability studies are of importance for the practical application of the model, and relevant theories have been established in \cite{Diethelm2002, Ford2013}.

	From the perspective of long-time dynamics, a fundamental question arises: whether this class of nonlocal evolution equations generate an autonomous dynamical system? If so, the asymptotic behavior of solutions as $t\to\infty$ could be obtained through attractor theory. However, the work \cite{Cong2017} revealed an important fact: such equations generally cannot generate a dynamical system in the classical sense, see also e.g. \cite{Doan2025}. Nevertheless, subsequent research has shown that they can be linked to a semi-dynamical system \cite{Doan2021}, which provides the possibility of applying infinite-dimensional dynamical systems theory.
	
	In this work, we extend the well-posedness  and stability results for semidynamical systems associated with equations involving Caputo fractional derivatives  \cite{Doan2021,Kloeden2023,Doan2024} to  the generalized Rayleigh-Stokes equation (\ref{eq:1.1}) on  $L^p(\mathbb{R}^d)$. Compared with classical models, the innovation of this paper lies in generalizing the specific fractional derivative $\partial_t^\alpha$ to a nonlocal derivative operator $\partial_t^{*,k}$, thereby describing a broader range of memory effects.
	The main contributions of this study are as follows. Firstly, we obtain well-posedness results for the generalized Rayleigh-Stokes equation. Secondly, by constructing appropriate dissipativity conditions under the Lipschitz condition, we establish the global existence and uniqueness of solutions. Furthermore, in infinite-dimensional spaces, we construct a fractional semi-dynamical system and prove the existence of attractors, thereby extending the corresponding results from finite-dimensional to infinite-dimensional settings. Finally, the semi-dynamical system we have constructed can be viewed as an endeavor to investigate the asymptotic behavior of solutions to nonlocal evolution equations using dynamical systems methods, and the introduction of a dynamical systems framework may offer fresh perspectives for the stability theory of nonlocal evolution equations.
	
	Due to the fact that nonlocal systems are unable to directly generate classical dynamical systems,
	the following points shed light on the challenges encountered and their corresponding solutions. Firstly, to address the difficulty that classical semigroup theory cannot directly handle the solvability of nonlocal equations, we embed the solution operator into the framework of semi-dynamical systems. This approach enables us to employ classical tools from dynamical systems theory to obtain the asymptotic properties of solutions. Secondly, nonlocal operators are extensions of fractional order operators, when it comes to handling nonlocal operators, we introduce an extended state space, this transformation turns the time convolution into a local operator in space, effectively revealing the corresponding dissipative behavior.
	Thirdly, to overcome the problem of applying the dissipativity condition to the vector field in the integral equation and defining a semigroup to establish the existence of an absorbing set in the space, we utilize a subspace of the space $\mathcal{C}$ along with a weighted norm.  Finally, due to the lack of compactness in $L^p(\mathbb R^d)$, and the Ascoli-Arzel\`{a} theorem cannot be directly applied to establish the compactness, we adopted a combination of finite dimensional approximation and the measure of noncompactness to obtain the compact results, thereby establishing the asymptotic compactness for requiring of attractors.

	The rest of the paper is organized as follows. In the next section, by employing the theory of complete positivity and the subordination principle, we derive an explicit expression for the mild solution of problem (\ref{eq:1.1}). In Section 3, using perturbation techniques within a Banach space equipped with an appropriate Bielecki weighted norm, we establish solvability results for the nonlinear problem and demonstrate continuous dependence of solutions on initial data. Furthermore, we show that the Volterra integral equation associated with the autonomous generalized Rayleigh-Stokes equation in $L^p(\mathbb{R}^{d})$ generates a semigroup on the space of continuous functions, which exhibits topological uniform convergence on compact subsets. When restricted to initial function $\varphi(t) \equiv S(t){u_0}$ for any $u_0 \in  L^p(\mathbb{R}^{d})$, this semigroup corresponds to the semi-dynamical system of the generalized Rayleigh-Stokes equation. In Section 4, we prove that if the vector field $ f $ is dissipative, the semigroup admits a generalized  attractor within the Banach subspace $\mathcal{C}_\alpha$.

	\section{Preliminaries}
	Throughout this paper, let $\mathbb R^+=(0,+\infty)$, and $*$ denotes the convolution of functions $f,g\in L_{loc}^1(\mathbb R_+)$ such that
	$$
	(f*g)(t)=\int_0^t f(t-s)g(s)ds.
	$$
	
	A kernel is said to be of type $(\mathcal{PC})$ if the kernel function $k\in L^1_{loc}(\mathbb R^+)$ is nonnegative and nonincreasing, and there exists another kernel $l\in L^1_{loc}(\mathbb R^+)$ such that $k*l=1$ on $\mathbb R^+$, see e.g. \cite{Vergara}, rewriting $(k,l)\in (\mathcal{PC})$. A $C^\infty$ kernel $k:\mathbb R^+\to \mathbb R $ is called completely monotonous if $(-1)^n k^{(n)}(t)\geq0$ for every $n\in\mathbb{N}_0$, $t>0$.  A $C^\infty$ kernel $k:\mathbb R^+\to \mathbb R $ is called a Bernstein function if it is nonnegative and its first derivative $k'(t)$ is completely monotonous. The classes of completely monotonous and Bernstein functions will be denoted by $\mathcal {CMF}$ and $\mathcal{BF}$, respectively. A characterization of completely monotonous kernel is given by Bernstein's theorem, which says that $f\in\mathcal{CMF}$ iff $f$ can be represented  as the Laplace transform of a nonnegative measure.
	A kernel $b\in L_{loc}^1(\mathbb R_+)$ is called completely positive if the solutions $s_\mu(t)$ and $r_\mu(t)$ of the following integral equations
	\begin{equation}\label{solution1}
		s_\mu +\mu b*s_\mu=1,~~~~~
		r_\mu +\mu b*r_\mu=b,
	\end{equation}
	are nonnegative for each $\mu >0.$
	
	Denote the Fourier transform of any function $ f $ by $ \tilde{f} $ and its Laplace transform by $ \hat{f} $.
	It is notice the fact that $ b $ is completely positive implying the function
	$$
	\varphi^b(\lambda): = \frac{1}{\widehat{b}(\lambda)}\in\mathcal{BF}, \quad \lambda > 0,
	$$
	Furthermore, for $ \tau \geq 0 $, define $ \psi_\tau^b(\lambda) = e^{-\tau \varphi^b(\lambda)} $. According to \cite[Proposition 4.5]{Pruss1993}, $ \psi_\tau^b \in \mathcal{CMF} $. Then by Bernstein's theorem, there exists a unique non-decreasing function $ \omega(\cdot,\tau) \in BV(\mathbb{R}_+) $ satisfying $ \omega(0,\tau) = 0 $, left-continuity, such that
	\begin{equation*}
		\widehat{\omega}(\lambda,\tau)=\int_{0}^{\infty}e^{-\lambda\varsigma}\omega(\varsigma,\tau)\:
		\mathrm{d}\varsigma=\frac{\psi_{\tau}^b(\lambda)}{\lambda},\quad\lambda>0,~~~~\tau\in\mathbb R_+.
	\end{equation*}
	It is clear that $w(\cdot,\tau)$ admits the semigroup property
	$$
	\int_0^t w(t-s,\tau_1)dw(s,\tau_2)=w(t,\tau_1+\tau_2),~~~~t,\tau_1,\tau_2\geq 0,
	$$
	and $w(t,0)=H_0(t)$, $t\geq0$ for the Heaviside function.
	The function $ w(t,\tau) $ is called the propagation function associated with the completely positive kernel $ b$. Moreover, $w(\cdot,\cdot)$ is Borel measurable on $\mathbb R_+\times\mathbb R_+$, $w(t,0)=1$ and $w(t,\infty)=0$.
	Let $b\in L^1_{loc}(\mathbb R^+)$ be of subexponential growth, which means
	$\int_0^\infty e^{-\epsilon t}|b(t)|dt <\infty$ for all $\epsilon > 0$. Moreover, $b$ is
	called $k$-regular if $|\lambda \hat{b}^{(n)}(\lambda)|\leq c |\hat{b} (\lambda)|$ for some constant $c > 0$, Re$\lambda>0$, $0\leq n\leq k$. From Proposition 3.5 and Proposition 4.9 in \cite{Pruss1993}, the solution $ s(t,\mu) $ of (\ref{solution1}) can be expressed as
	$$
	s(t,\mu) = -\int_{0}^{\infty} e^{-\mu \tau} d_\tau w(t,\tau), \quad t > 0,\ \mu > 0,
	$$
in particular
	$$
	-\int_{0}^{\infty} d_\tau w(t,\tau) = 1.
	$$
	Additionally, using the integration by parts, the Laplace transform of $s(t,\mu) $ is given by
	$$
	\hat{s}(\lambda,\mu)=\left(-\int_0^\infty e^{-\mu\tau}d_\tau w(t,\tau)\right)^{\wedge}(\lambda)= \frac{1}{\lambda(1+\mu\hat{b}(\lambda))}.
	$$

	Now, let's turn back to our problem (\ref{eq:1.1}).
	Assume that $ Z(t,x) $ is the fundamental solution of problem
	\begin{equation*}
		\partial_tZ-(1+\partial_t^{*,k})\Delta Z=0,\quad ~~t>0,\ ~~x\in\mathbb{R}^d,
	\end{equation*}
	satisfying $Z|_{t=0}=\delta_0$, by employing the theory of completely positive and the subordination principle considered in \cite{Peng2026}, and applying the Fourier transform with respect to $ x $, and the Laplace transform with respect to $t$, we find that $\hat{\tilde{Z}}=\lambda^{-1}(1+|\xi|^2\hat{m})^{-1}\delta_0$, and then the solution is given by
	$
	\tilde{Z}(t,\xi) = s(t, |\xi| ^2)\delta_0 $, for $t \geq 0,$ $\xi \in \mathbb{R}^d,$
	where $ s(t, |\xi| ^2)  $ is the unique solution to equation (\ref{solution1}), $m:=1+k $ for kernel $k\in L_{loc}^1(\mathbb R_+)$.
	Obviously,
	$$
	s(t, |\xi| ^2)=-\int_0^\infty e^{-t|\xi|^2}d_\tau\omega(t, \tau)= \tilde{Z}(t,\xi) .
	$$
	Taking the inverse Fourier transform into $Z$ with respect to $ \xi $, we get the solution $ Z(t,x)$ be defined by
	\begin{equation*}
		Z(t,x)=-\int_{\mathbb R^d}\int_0^\infty G(\tau,x-y)\delta_0(y)d_\tau\omega(t, \tau)dy,\quad t>0, ~~x\in\mathbb{R}^d,
	\end{equation*}
	where $G$ is the Gaussian heat kernel.
	
	It should be noted that, if $k\in\mathcal{CMF}$,  then $m \in\mathcal{CMF}$ as well. As mentioned in \cite{Clement1981,Miller1968}, this property ensures the complete positivity of function $m$.  For more properties of the relaxation function, see references \cite{Ke2020, Gripenberg1990, Pruss1993}. However, for a generalization term $\partial_tu$ in (\ref{eq:1.1}), the kernel $k$ will contain  the Dirac delta distribution that does not match the $\mathcal{PC}$-type. The current problem does not fit  uniformly into the $\mathcal{PC}$-type framework, for example, the stability and semigroup structure of semi-dynamical systems were studied in \cite{Nguyen2024}, the stability, instability, and blowup solutions were investigated in \cite{Vergara}.

	Using the resolvent family and the subordination principle, let $-A=\Delta$ on $L^p(\mathbb R^d)$ for $1\leq p<\infty$ with norm $\|\cdot\|$, then $-A$ is the infinitesimal generator of the heat semigroup $\mathcal{Q}(t)=e^{-At}$, which is analytic and contractive semigroup on the maximal domain $D(A)=\{v\in L^p(\mathbb R^d),-A v\in L^p(\mathbb R^d)\}$. The resolvent set $\rho(-A)\subset \mathbb C^+$ if $1<p<\infty$ and there exists a constant $M\geq1$ such that $\|(\lambda+A)^{-1}\|\leq M |\lambda|^{-1}$ for Re$\lambda>0$.
	
	Consider the homogeneous equation of  problem (\ref{eq:1.1}), then it is equivalent to the Volterra-type integral equation
	\begin{equation}\label{equivalent}
		u+m* A u =u_0,
	\end{equation}
	which implies that $\hat{u}(\lambda)=\lambda^{-1}(1+\hat{m}A)^{-1}u_0.$
	
	Recall that
	$$
	(I+\lambda A)^{-1}=\lambda^{-1}\int_0^\infty e^{-\tau/\lambda} \mathcal{Q}(\tau)d\tau,~~~~{\rm Re}  \lambda>0,
	$$
	and then
	$$
	\hat{S}(\lambda)u_0=\lambda^{-1}(1+\hat{m}A)^{-1}u_0=\lambda^{-1}\varphi^m(\lambda)\int_0^\infty e^{-\tau \varphi^m(\lambda)} \mathcal{Q}(\tau)u_0d\tau=:\int_0^\infty h(\lambda, \tau) \mathcal{Q}(\tau)u_0d\tau.
	$$
Since $\hat{h}(\lambda,\cdot)(\mu)=\hat{s}(\lambda,\mu)$ and $\left(-d_\tau \hat{w}(\lambda,\tau)\right)^{\wedge}(\mu)=\hat{s}(\lambda,\mu)$, the uniqueness of Laplace transform implies that
	$h(\lambda, \tau)=-d_\tau \hat{w}(\lambda,\tau)$. Therefore, we have
	$$
	\hat{S}(\lambda)u_0=- \int_0^\infty \mathcal{Q}(\tau)u_0 d_\tau \hat{w}(\lambda,\tau)=\left( -\int_0^\infty \mathcal{Q}(\tau)u_0d_\tau w (t,\tau) \right)^{\wedge}(\lambda).
	$$
	Consequently, by the uniqueness of Laplace transform, we get the solution by
	\begin{equation*}
		u(t)= S(t)u_0=-\int_0^\infty \mathcal{Q}(\tau)u_0d_\tau w (t,\tau) .
	\end{equation*}
	
	By the variation of parameters formula for Volterra equations, the solution of  (\ref{eq:1.1}) is given by
	\begin{equation}\label{Mild-solution}
		\begin{aligned}
			u(t)= S(t)u_0+\int_0^tS(t-s) f(u(s))ds.\end{aligned}
	\end{equation}

	Following these discusses, we say, a function $u$ is called a mild solution to problem (\ref{eq:1.1}) if $u\in C([0,T];L^p(\mathbb R^d))$ for every $T>0$ and it satisfies
	(\ref{Mild-solution}). In order to capture the dynamic properties, we introduce a generalized mild solution $u\in  C([0,T];L^p(\mathbb R^d))$ for every $T>0$ satisfying
	\begin{equation}\label{gene-Mild-solution}
		u(t)=\varphi(t)+\int_0^tS(t-s) f(u(s)) {d}s,\quad t>0,
	\end{equation}
	for some given function $\varphi\in BC(\mathbb R_+,L^p(\mathbb R^d)).$
	Clearly, the mild solution in (\ref{Mild-solution})  is a special case to (\ref{gene-Mild-solution}) when $\varphi(t)=S(t)u_0.$

	\section{Well-posedness and stability}\label{sect:3}
	In this section, we first establish the well-posedness of solutions to problem (\ref{eq:1.1}).
	
	Throughout this paper, we  assume that the kernel function $k\in L_{loc}^{1}(\mathbb{R}^{+})$ is completely monotonic and possibly unbounded on $\mathbb{R}^+$. Further, {we suppose there exist $\phi \in (\frac{\pi}{2},\pi)$ and $\delta \in (0,\pi)$ such that $\frac{1}{\widehat{m}(\lambda)} \notin (-\infty, 0]$ and $|\arg \widehat{m}(\lambda)| \leq \delta$ for $\lambda \in \{\lambda \in \mathbb C: \lambda \neq 0, |\arg(\lambda)| \leq \phi\}$.} Moreover, the nonlinear function should be satisfied the following assumption.
	\begin{itemize}
		\item [{\rm $(F_f)$}] The function $f$ satisfies $f(0)=0$ and is globally Lipschitz continuous, i.e., there exists a positive constant $L_f$ such that
		$$
		\| f(u)-f(v)\|  \leq L_f \| u - v\| \quad \text{for all } u, v \in L^p(\mathbb{R}^d).
		$$
	\end{itemize}
	
	\subsection{Well-posedness of solutions}
	Utilizing perturbation technique and working in a Banach space furnished with a suitable Bielecki weighted norm, we obtain the well-posedness results of problem (\ref{eq:1.1}).
	
	\begin{lemma}\label{le:3.1}
		For any $v\in$ $L^p(\mathbb{{R}}^{d})$, $S( \cdot ) v \in C( [ 0, \infty); L^p(\mathbb{{R}}^{d}))$. Moreover, $\|S(t)v\|\leq \|v\|$, for all $t\geq 0$.
	\end{lemma}
	
	\begin{proof}
		For any \( v \in L^p(\mathbb{R}^d) \), for each \( t_1, t_2\geq 0 \), we have
		\[
		\|S(t_1)v - S(t_2)v\|  = \left\| \int_0^\infty \mathcal{Q}(\tau)v d_\tau\, (\omega(t_2, \tau) - \omega(t_1,  \tau)) \right\| .
		\]
		Since \( \omega(t, \tau) \) is the propagation function associated with a completely positive kernel \( m \), i.e., for all \( t \geq0 \), we have $-\int_{0}^{\infty}d_\tau\omega(t,  \tau) = 1$, the family of measures $\{{d_{\tau}\omega(t,\cdot)}\}_{t\geq0}$ is tight. Thus, for any \(\varepsilon > 0\), there exists \( N > 0 \) such that
		\[
		-\int_N^\infty  d_\tau \omega(t,  \tau)  < \varepsilon, \quad \forall t \geq0.
		\]
		Since $\widehat{\omega}(\lambda,\tau)=\frac{e^{-\tau\varphi^m(\lambda)}}{\lambda}$, where $\varphi^m(\lambda)$ is a Bernstein function, and the family of measures $\{d_{\tau}\omega(t,\tau)\}$ is uniformly bounded and tight, according to the generalized continuity theorem, for any continuous and bounded function $f(\tau)$, the mapping
		$$t\mapsto \int_{0}^{N}f(\tau)\,d_{\tau}\omega(t,\tau)
		$$
is continuous.

		Splitting the integral into the intervals \([0, N]\) and \((N, \infty)\), we get
		\[
		\begin{aligned}
			\|S(t_1)v - S(t_2)v\| &\leq \left\| \int_0^N \mathcal{Q}(\tau)vd_\tau (\omega(t_2, \tau) - \omega(t_1,  \tau))\right\| \\
			&\quad +  \left\|\int_N^\infty \mathcal{Q}(\tau)v d_\tau \omega(t_2,  \tau) \right\| +\left\| \int_N^\infty \mathcal{Q}(\tau)v d_\tau\omega(t_1, \tau)\right\| .
		\end{aligned}
		\]
		For the first part, since the mapping $\tau\mapsto \mathcal{Q}(\tau)v$ is strongly continuous from $[0,\infty)$ to $L^p(\mathbb{R}^d)$. Moreover, since \(|d_\tau\, (\omega(t_2, \tau) - \omega(t_1,  \tau))|\) denotes the total variation of the difference of the measures,
		Therefore, when $t_1\to t_2$, the integral difference tends to 0. In summary, using $\|T(t)v\|\leq \|v\|$,  there is a $\delta>0$ such that for $|t_2-t_1|<\delta$,
		$$\|S(t_1)v-S(t_2)v\| \leq  3\|v\|\varepsilon.$$
		Hence, \( S(\cdot) v \) is strongly continuous with respect to \( t \) for any $v\in L^p(\mathbb{R}^d)$, i.e., \( S(\cdot) v \in C([0, \infty); L^p(\mathbb{R}^d)) \).  Furthermore, using Minkowski's integral inequality, we have
		\[
		\|S(t)v\|  \leq -\int_0^\infty \|\mathcal{Q}(\tau)v\| \omega(t, \mathrm{d}\tau) \leq\|v\|,~~~\forall t\geq0.
		\]
		The proof is completed.
	\end{proof}
	
	For any $\gamma>0$, let $\mathcal{C}_\gamma$ denote the Banach space of all continuous functions $w\in C([0,\infty);L^p(\mathbb{{R}}^{d}))$ endowed with the Bielecki weighted norm
	$$\|w\|_{\mathcal{C}_\gamma }=\sup_{t\in[0,\infty)}e^{-\gamma t} \|w(t)\| .$$
	
	\begin{theorem}\label{thm:1.1}
		Assume that ($F_f$) is fulfilled. Then for given continuous function $\varphi\in BC([0,\infty);L^p(\mathbb{{R}}^{d})),$ the equation (\ref{gene-Mild-solution}) has a unique generalized mild solution $u(t,\varphi)$ on $[ 0,\infty) $. Furthermore, for $\gamma>L_f$, it yields
		
		$$\|u(t,\varphi)\| \leq M_\varphi e^{L_f t},\quad t\geq0,$$
		where $M_\varphi=\sup_{t\in[0,\infty)} \|\varphi(t)\| $. Moreover, the solution mapping $\varphi\mapsto u(t,\varphi)$ corresponding to (\ref{gene-Mild-solution}) is continuous  in the sense of
		$$\|u(t,\varphi)-u(t,\bar{\varphi})\| \leq e^{ L_f t} \|\varphi(t)-\bar{\varphi}(t)\| .$$
	\end{theorem}
	\begin{proof}
		
		We define an operator $\Phi$ on $C([0,\infty);L^p(\mathbb{R}^{d}))$ by
		$$\Phi(u)(t)=\varphi(t)+\int_0^tS(t-s) f(u(s))ds.$$
		By the properties of the operators $S(t)$, one concludes that $\Phi$ acts on $C([0,\infty);L^p(\mathbb{R}^{d})).$
		
		Let $u,v$ be in ${\mathcal{C}_\gamma}$. Because $f(\cdot)$ is Lipschitz continuous, we have
		$$\begin{aligned}\|\Phi(u)(t)-\Phi(v)(t)\|
			&= \left\|\int_{0}^{t}S(t-\tau)(f(u(\tau))-f(v(\tau)))d\tau\right\| \\
			&\leq L_{f}\int_{0}^{t} \|u(\tau)-v(\tau)\| d\tau\\
			&\leq\:\frac{L_f}{\gamma}\:\sup_{\tau\in[0,\infty)}e^{-\gamma \tau} \|u(\tau)-v(\tau)\|  ( e^{\gamma t}-1).\end{aligned}$$
		Consequently, $ \|\Phi(u) -\Phi(v) \|_{\mathcal{C}_\gamma} \leq {L_f} \|u-v\|_{\mathcal{C}_\gamma}/{\gamma}.$
		It yields that $\Phi$ is a contraction mapping in ${\mathcal{C}_\gamma}$ in view of $L_f<\gamma.$ Using $f(0)=0$, similarly we get that $\Phi: \mathcal C_\gamma\to \mathcal C_\gamma$. Note that $\mathcal{C}_\gamma$ is a subset space of $C([0,\infty);L^p(\mathbb{R}^{d}))$,
		Hence, the Banach fixed point theorem implies that $\Phi$ has a unique fixed point $u$ belonging to $C([0,\infty);L^p(\mathbb{R}^{d})).$
		
		Moreover, by the definition of $\Phi(u)$, we can rewrite
		$$u(t,\varphi)=\varphi(t)+\int_0^tS(t-\tau)\left(f(u(\tau,\varphi))-f(0)\right)d\tau,$$
		which together with Lemma \ref{le:3.1} implies that
		$$\|u(t,\varphi)\| \leq\:\|\varphi(t)\| +L_{f}\int_{0}^{t} \|u(\tau,\varphi)\| d\tau \leq\:M_{\varphi} +L_{f}\int_{0}^{t}\|u(\tau,\varphi)\| d\tau.$$
We apply the integral form of Gronwall's inequality to obtain
$$\|u(t,\varphi)\| \leq M_\varphi e^{L_f t},$$
		In particular, when $\varphi(t)=S(t)u_0$, it holds for a mild solution to (\ref{eq:1.1}) as
	 $\|u(t )\|  \leq\:e^{L_f t}\|u_0\|$, $t\geq0.$

		Let $\varphi,\bar{\varphi}\in BC([0,\infty);L^p(\mathbb{R}^{d})).$  Using a similar argument in this proof, one has
		$$\begin{aligned}\|u(t,\varphi)-u(t,\bar{\varphi})\| &\leq\:\|\varphi(t)-\bar{\varphi}(t)\| +L_{f}\int_{0}^{t}\|u(\tau,\varphi)-u(\tau,\bar{\varphi})\| d\tau \end{aligned}$$
		By using again with the Gronwall inequality,
		one concludes that
		$$\|u(t,\varphi)-u(t,\bar{\varphi})\| \leq e^{L_f t}\|\varphi(t)-\bar{\varphi}(t)\| .$$
		The proof is completed.

	\end{proof}

	\subsection{Continuity of kernel functions}
	
	\begin{lemma}\label{le:3.2}
		For any $\eta,\varepsilon>0$, $\lambda>0$ there exists $\delta(\lambda)>0$ with $\delta(\lambda)\to0$ as $\lambda\to0$ such that for all real completely positive kernels $m_,\bar{m}\in L^1_{loc}(\mathbb R^+)$ satisfying 
$|\hat{m}-\hat{\bar{m}}|  \leq \delta(\lambda)$ it holds
		$$\int_0^t \|S(\tau)-\bar{S}(\tau)\| d\tau < e^{\eta t} \varepsilon.$$
		where $S,\bar{S}$ are the solution operators with respect to the kernels $ m, \bar{m},$ respectively.
	\end{lemma}
	\begin{proof}
		Let $\|\omega(t,\cdot)-\bar{\omega} (t,\cdot)\|_{TV}$ be the total variation measure of the difference of the propagation functions given by
		$$\|\omega(t,\cdot)-\bar{\omega} (t,\cdot)\|_{TV}:=\int_{0}^{\infty} |d_\tau \left(\omega(t,\tau)-\bar{\omega} (t,\tau)\right)|.$$
		
		For $\lambda>0$, since the Laplace transforms to the difference
		\[
		\widehat{\omega}(\lambda, \tau) - \widehat{\bar{\omega}}(\lambda, \tau) = \int_0^\infty e^{-\lambda t} (\omega(t, \tau) - \bar{\omega}(t, \tau)) \, dt,
		\]
		taking the total variation measure with respect to $\tau$ gives
		\[
		\int_0^\infty |d_\tau (\widehat{\omega}(\lambda, \tau) - \widehat{\bar{\omega}}(\lambda, \tau))| = \int_0^\infty \left| d_\tau \left( \int_0^\infty e^{-\lambda t} (\omega(t, \tau) - \bar{\omega}(t, \tau)) \, dt \right) \right|.
		\]
		Using Tonelli's theorem to interchange the order of integration, we have
		\[
		\int_0^\infty |d_\tau (\widehat{\omega}(\lambda, \tau) - \widehat{\bar{\omega}}(\lambda, \tau))| = \int_0^\infty e^{-\lambda t} \left( \int_0^\infty |d_\tau (\omega(t, \tau) - \bar{\omega}(t, \tau))| \right) dt.
		\]
		Hence, we have
		\[
		\int_0^\infty e^{-\lambda t} \|\omega(t, \cdot) - \bar{\omega}(t, \cdot)\|_{TV} \, dt = \int_0^\infty |d_\tau \left(\widehat{\omega}(\lambda, \tau) - \widehat{\bar{\omega}}(\lambda, \tau)\right)|.
		\]

		In view of $\psi_\tau^m(\lambda) - \psi _\tau^{\bar{m }}(\lambda)=e^{-\tau\varphi^m(\lambda)}-e^{-\tau\varphi^{\bar{m }}(\lambda)}$, each $\tau\in\mathbb R^+,$  we have
		\begin{align*}
			d_\tau(\widehat{\omega}(\lambda,\tau) - \widehat{\bar{\omega}}(\lambda,\tau))=(\varphi ^m(\lambda) -  \varphi  ^{\bar{m }}(\lambda))\widehat{\omega}(\lambda,\tau)+\varphi ^{\bar{m }}(\lambda)\frac{\psi_\tau^m(\lambda) - \psi _\tau^{\bar{m }}(\lambda) }{\lambda}.
		\end{align*}
		Since $\varphi^{m }$ and $\varphi^{\bar{m }}$ are real Bernstein functions for $\lambda>0$, for the case if $\varphi^m(\lambda)>\varphi^{\bar{m }}(\lambda)$, using the mean value theorem,  then
		$$e^{-\tau\varphi^m(\lambda)}-e^{-\tau\varphi^{\bar{m }}(\lambda)}=\tau e^{-\tau\xi(\lambda)}(\varphi^{m }(\lambda)-\varphi^{\bar{m }}(\lambda)),$$
		where $\xi(\lambda)$ lies in $[\varphi^{\bar{m }}(\lambda),\varphi^{m }(\lambda)]$. By virtue of  $\varphi^{m }(\lambda)\geq0$ and $\varphi^{\bar{m }}(\lambda)\geq0$, we have $\xi(\lambda)\geq0$. Then
		$$ \psi_\tau^m(\lambda) -  \psi _\tau^{\bar{m }}(\lambda) =\tau e^{-\tau\xi(\lambda)} (\varphi^{m }(\lambda)-\varphi^{\bar{m }}(\lambda) )\leq \tau e^{-\tau\varphi^{\bar{m }}(\lambda)}(\varphi^{m }(\lambda)-\varphi^{\bar{m }}(\lambda)).$$
		Thus,
		\begin{align*}
			d_\tau(\widehat{\omega}(\lambda,\tau) - \widehat{\bar{\omega}}(\lambda,\tau))\leq &(\varphi ^m(\lambda) -  \varphi  ^{\bar{m }}(\lambda))\widehat{\omega}(\lambda,\tau)+\varphi ^{\bar{m }}(\lambda)\frac{\tau e^{-\tau\varphi^{\bar{m }}(\lambda)}(\varphi^{m }(\lambda)-\varphi^{\bar{m }}(\lambda))}{\lambda}\\
			=&(\varphi ^m(\lambda) -  \varphi  ^{\bar{m }}(\lambda))(\widehat{\omega}(\lambda,\tau)+\tau\varphi ^{\bar{m }}(\lambda)\widehat{\bar{\omega}}(\lambda,\tau)).
		\end{align*}
		For the another case if $\varphi^m(\lambda)\leq \varphi^{\bar{m }}(\lambda)$, similarly using $1-e^{-\tau x}\leq \tau x$ for $x>0$, we have
		\begin{align*}
			d_\tau(\widehat{\omega}(\lambda,\tau) - \widehat{\bar{\omega}}(\lambda,\tau))\leq &(\varphi ^m(\lambda) -  \varphi  ^{\bar{m }}(\lambda))\widehat{\omega}(\lambda,\tau)+\varphi ^{\bar{m }}(\lambda)\frac{\tau e^{-\tau\varphi^{ m }(\lambda)}(\varphi  ^{\bar{m }}(\lambda)-\varphi^{m }(\lambda) )}{\lambda}\\
			=&(\varphi ^m(\lambda) -  \varphi  ^{\bar{m }}(\lambda))\widehat{\omega}(\lambda,\tau)+\tau\varphi ^{\bar{m }}(\lambda)\widehat{ \omega }(\lambda,\tau)(\varphi  ^{\bar{m }}(\lambda)-\varphi^{m }(\lambda) ).
		\end{align*}
		Therefore, we have
		\begin{align*}
			|d_\tau(\widehat{\omega}(\lambda,\tau) - \widehat{\bar{\omega}}(\lambda,\tau))|\leq & |\varphi ^m(\lambda) -  \varphi  ^{\bar{m }}(\lambda)| (\widehat{\omega}(\lambda,\tau)+\tau\varphi ^{\bar{m }}(\lambda)\widehat{ \omega }(\lambda,\tau)+\tau\varphi ^{\bar{m }}(\lambda)\widehat{ \bar{\omega} }(\lambda,\tau)).
		\end{align*}

 From the assumptions of $m,\bar{m}$, let $\lambda >0$, the difference of their Laplace transforms is $|\hat{m}-\hat{\bar{m}}|  \leq \delta(\lambda)$.
			Therefore, in view of
$$
\hat{m}(\lambda)\geq \int_0^1 e^{-\lambda s} m(s)ds\geq e^{-\lambda }\int_0^1 m(s)ds=:e^{-\lambda }m_c,
$$
similar $\hat{\bar{m}}(\lambda)\geq e^{-\lambda }\bar{m}_c$, then
			$$|\varphi^m(\lambda) - \varphi^{\bar{m }}(\lambda)| = \left|\frac{1}{\hat{m}(\lambda)} - \frac{1}{\hat{\bar{m}}(\lambda)}\right| \leq \frac{\delta(\lambda)e^{2\lambda}}{m_c\bar{m}_c}.$$

			Using these inequalities and the estimate for the difference of $|d_\tau(\widehat{\omega}(\lambda,\tau) - \widehat{\bar{\omega}}(\lambda,\tau))|$, we obtain
			
			$$\int_{0}^{\infty}e^{-\lambda t}\|\omega(t,\cdot)-\bar{\omega}(t,\cdot)\|_{TV}dt \leq \frac{\delta(\lambda)e^{2\lambda}}{m_c\bar{m}_c}\int_{0}^{\infty}(\widehat{\omega}(\lambda,\tau)+\tau\varphi ^{\bar{m }}(\lambda)\widehat{ \omega }(\lambda,\tau)+\tau\varphi ^{\bar{m }}(\lambda)\widehat{ \bar{\omega} }(\lambda,\tau))d\tau.$$
			Evaluating the integrals
			$$\int_{0}^{\infty}  e^{-\tau x} d\tau=\frac{1}{x },~~\int_{0}^{\infty}\tau e^{-\tau x} d\tau=\frac{1}{x^{2} },~~x>0,$$
			it follows that
			$$\int_{0}^{\infty}e^{-\lambda t}\|\omega(t,\cdot)-\bar{\omega}(t,\cdot)\|_{TV} dt \leq \frac{\delta(\lambda)e^{2\lambda}}{\lambda m_c\bar{m}_c}\left(\frac{1}{ \varphi ^{ m  }(\lambda) }+\frac{\varphi ^{\bar{m }}(\lambda)}{\left(\varphi ^{ m  }(\lambda)\right)^2}+\frac{1}{\left(\varphi ^{\bar{m }}(\lambda)\right)^2}\right).$$

			By the representation formula of solution operator, and $\omega,\bar{\omega}$ are the propagation functions with respect to the kernels $ m, \bar{m},$ respectively, one has
			$$\begin{aligned}\|S(t)-\bar{S}(t)\| &=\left\|-\int_0^\infty \mathcal{Q}(\tau)d_\tau(\omega(t, \tau)- \bar{\omega}(t, \tau))\right\|  \leq \|\omega(t,\cdot)-\bar{\omega}(t,\cdot)\|_{TV}.
			\end{aligned}$$
			In view of $\varphi^{m },\varphi^{\bar{m }}\in \mathcal{BF}$ with respect to kernels $m,\bar{m}$, it thus is continuous for $\lambda=\eta$, for any $\eta>0$, we have
			\begin{align*}
				\int_0^t \|S(\tau)-\bar{S}(\tau)\|  d\tau \leq & e^{\eta t}\int_0^t    e^{-\eta \tau} \|\omega(\tau,\cdot)-\bar{\omega}(\tau,\cdot)\|_{TV}d\tau\\
				\leq& e^{\eta t}\int_0^\infty    e^{-\eta \tau} \|\omega(\tau,\cdot)-\bar{\omega}(\tau,\cdot)\|_{TV}d\tau\\
				\leq& e^{\eta t}\left(\frac{\delta(\lambda)e^{2\lambda}}{\lambda m_c\bar{m}_c} \left(\frac{1}{ \varphi ^{ m  }(\lambda) }+\frac{\varphi ^{\bar{m }}(\lambda)}{\left(\varphi ^{ m  }(\lambda)\right)^2}+\frac{1}{\left(\varphi ^{\bar{m }}(\lambda)\right)^2}\right)\right)\bigg|_{\lambda=\eta}.
			\end{align*}
			
			Now, let
			$$
			c(\eta)=\frac{ 1}{\lambda    }\left(\frac{1}{ \varphi ^{ m  }(\lambda) }+\frac{\varphi ^{\bar{m }}(\lambda)}{\left(\varphi ^{ m  }(\lambda)\right)^2}+\frac{1}{\left(\varphi ^{\bar{m }}(\lambda)\right)^2}\right) \bigg|_{\lambda=\eta},
			$$
			for any $\varepsilon>0$, choose $\delta(\eta) = m_c\bar{m}_c\varepsilon/  c(\eta)$. When $|\hat{m}-\hat{\bar{m}}|  \leq \delta(\eta)$, it holds that
			$$\int_0^t \|S(\tau)-\bar{S}(\tau)\|  d\tau < e^{\eta t} \varepsilon.$$
			This completes the proof.
		\end{proof}

Obviously, the above Theorem \ref{thm:1.1} and Lemma \ref{le:3.2} hold for $t\in[0,T]$ with every $T>0$ and $\varphi\in C( [ 0, T] ;L^p(\mathbb{R}^{d}))$. Accordingly, we get the stability of kernel function.
		\begin{theorem}
			Assume that ($F_f$) is fulfilled. Let $u_m(t,\varphi)$ denote the unique solution to (\ref{gene-Mild-solution}) corresponding to the kernel $m$ and the initial condition $\varphi.$ Then the solution mapping $(\varphi,m)\mapsto u_m(t,\varphi)$ is continuous from $C( [ 0, T] ;L^p(\mathbb{R}^{d})) \times L^{1}( 0, T) $ into $C( [ 0, T] ; L^p(\mathbb{R}^{d})) .$
		\end{theorem}
		\begin{proof}
			Observe that the estimate and Theorem \ref{thm:1.1} means that the solution mapping is the composition of bounded maps, so it is also bounded.
			Fix a point $(\varphi,m)$ and consider $(\bar{\varphi},\bar{m})$ belonging to a neighbourhood of $(\varphi,m).$ Let $u_m(t,\varphi)$, $\bar{u}_{\bar{m}}(t,\bar{\varphi})$ be the unique mild solutions to (\ref{gene-Mild-solution}) with the kernels $m$ and $\bar{m}$ and the initial conditions $\varphi(t)$ and $\bar{\varphi}(t)$, respectively.
			
			Let $\bar{S}$ denote  the solution operators corresponding to the kernel $\bar{m}.$ By the formulae determining a mild solution to (\ref{gene-Mild-solution}), we have
			$$\begin{aligned}\|u_{m}(t,\varphi)-\bar{u}_{\bar{m}}(t,\bar{\varphi})\|
				&=\:\left\|\varphi(t)-\bar{\varphi}(t)+\int_{0}^{t}\Big[S(t-\tau)f(u(\tau))-\bar{S}(t-\tau)f(\bar{u}(\tau))\Big]d\tau\right\| \\&\leq\:\|\varphi(t)-\bar{\varphi}(t)\| +\int_{0}^{t} \|f(u(\tau))-f(\bar{u}(\tau))\| d\tau\\
				&\quad+\int_{0}^{t}\|S(t-\tau)-\bar{S}(t-\tau)\| \|f(\bar{u}(\tau))\| d\tau.
			\end{aligned}$$
			This means from Theorem \ref{thm:1.1} that
			$$\begin{aligned}\|u_{m}(t,\varphi)-\bar{u}_{m}(t,\bar{\varphi})\| &\leq\:\|\varphi(t)-\bar{\varphi}(t)\| +L_{f}\int_{0}^{t}\|u_{m}(\tau,\varphi)-\bar{u}_{\bar{m}}(\tau,\bar{\varphi})\| d\tau\\&\quad+L_fM_\varphi e^{ L_f t}\int_{0}^{t}\|S(\tau)-\bar{S}(\tau)\| d\tau  , \end{aligned}$$
			where from the assumption of $(F_f)$, we have used
			$$ \|f(\bar{u}(t) )\| \leq L_{f}  \|\bar{u}_{\bar{m}}(t,\bar{\varphi})\| \leq L_f(M_\varphi+1) e^{L_f t}  .$$
			
			For any $\varepsilon>0$. By the assumptions of $\varphi,\bar{\varphi}$, there exists a positive number $ \delta_{0}=\delta_0(\varepsilon)$ such that whenever $\|\varphi-\bar{\varphi}\|_{C}\leq\delta_{0} $. Let $ \varepsilon$ be an arbitrary positive number. Thanks to Lemma \ref{le:3.2} by choosing $\eta=\gamma-L_f$, there is a $\delta_1=\delta_1(\varepsilon)>0$ such that
			$$\|m-\bar{m}\|_{L^1(0,T)}\leq\delta_1\quad\text{yields}
\quad\int_{0}^{t}\|S(\tau)-\bar{S}(\tau)\| d\tau \leq e^{(\gamma-L_f)t}\varepsilon .$$
			Set $e(t)=\|u_{m}(t,\varphi)-\bar{u}_{\bar{m}}(t,\bar{\varphi})\| .$ Plugging the estimate above, one gets
			$$e(t)\:\leq\:\|\varphi(t)-\bar{\varphi}(t)\| +L_fM_\varphi e^{ \gamma t}\varepsilon +L_{f}\int_{0}^{t}e(\tau)d\tau.$$
			This means that for choosing $\delta=\min\{\delta_0,\delta_1\}$, we have
			$$\|u_m(\cdot,\varphi)-u_{\bar{m}}(\cdot,\bar{\varphi})\| \leq \left(\|\varphi(t)-\bar{\varphi}(t)\| +L_fM_\varphi e^{ \gamma t}\varepsilon \right)e^{L_f t}.$$
			This finishes the proof.
		\end{proof}

		\subsection{Semi-dynamical}\label{sect:3.3}
	We now  consider the infinite dimensional case with Lipschitz nonlinearity. Let  $\mathcal{C}$ be the Banach space of continuous functions $\varphi:\mathbb{R}^+\to L^p(\mathbb{R}^{d})$ with the topology uniform convergence on compact subsets. This topology is induced by the metric
		$$\rho(\varphi,h):=\sum_{n=1}^{\infty}\frac{1}{2^{n}}\rho_{n}(\varphi,h),$$
		where
		$$\rho_n(\varphi,h):=\frac{\sup_{t\in[0,n]}\|\varphi(t)-h(t)\| }{1+\sup_{t\in[0,n]}\|\varphi(t)-h(t)\| }.$$
		
		We follow Chapter XI, pp. 178-179, in Sell \cite{Sell1971} closely, simplifying it to this  autonomous  case, and show that the singular Volterra integral equation (\ref{gene-Mild-solution}) generates an autonomous semi-dynamical system on the space $\mathcal{C}.$
		For a given Lipschitzian nonlinearity $f\in\mathcal{C}$, we define the operator $\mathcal{T}_t:\mathcal{C}\to\mathcal{C}$ for each $t\in\mathbb{R}^+$, $\theta\in\mathbb{R}^+,$ by
		\begin{equation}\label{semigroup}
			(\mathcal{T}_t \varphi)(\theta)=\varphi(t+\theta)+\int_0^tS(t+\theta-\tau) f(\tau,u_\varphi)\:\mathrm{d}\tau,\quad\varphi \in\mathcal{C}([0,+\infty);L^p(\mathbb{R}^{d})),
		\end{equation}
		where $u_\varphi$ is a unique solution of the Volterra integral equation (\ref{gene-Mild-solution}) for this $\varphi$, i.e.,
		\begin{equation*}
			u_\varphi(t)=\varphi(t)+\int_0^tS(t-s)f(s,u_\varphi)ds.
		\end{equation*}

By employing an argument similar to that of Theorem 4 in  \cite{Ford2013}, the following lemma can be obtained.
		\begin{theorem}
Assume that ($F_f$) holds. Then the family of operators $\{\mathcal{T}_t\}_{t\geq0}$ associated to problem (\ref{gene-Mild-solution}) forms a semigroup of Lipschitz  mapping in $\mathcal{C}([0,+\infty);L^p(\mathbb{R}^{d}))$ with respect to the metric $\rho,$ i.e., $\mathcal{T}_{t+s}=\mathcal{T}_t\circ \mathcal{T}_s$ for any $t,s>0$ and $\mathcal{T}_0=I$.
		\end{theorem}
		
\begin{remark}
We observe that the semigroup structure of the semi-dynamic system indeed lies on $(\mathcal C,\rho)$, rather than on the usual space $\mathcal C$ with norm $\|\cdot\|$. The semigroup law at initial point $t,s=0$ is invalid unless $S(t)$ reduces to a classical $C_0$-semigroup.
\end{remark}		
		
		\section{Attractors of semi-dynamical systems}\label{sect:4}
Section \ref{sect:3} proved that the Volterra integral equation (\ref{gene-Mild-solution}), which is associated with autonomous generalized Rayleigh-Stokes equations in $L^p(\mathbb{R}^{d})$, generates a semigroup on the space $\mathcal{C}$, where $\mathcal{C}$ is the space of continuous functions
$\varphi:\mathbb{R}^+\to L^p(\mathbb{R}^{d})$
 endowed with the topology of uniform convergence on compact subsets. When restricted to initial functions of the form $\varphi(t) \equiv S(t){u_0}$
 for $u_0 \in  L^p(\mathbb{R}^{d})$, this semigroup serves as a semi-dynamical system for the generalized Rayleigh-Stokes equations.

The aim of this section is to show that when the vector field
$f$ possesses a dissipativity property, the semigroup does have a global generalized attractor in a subspace $\mathcal{C}_\alpha$ of $\mathcal{C}([0,+\infty);L^p(\mathbb{R}^{d}))$. This attractor is quite unusual because it only attracts a restricted class of initial values. Moreover, it must be defined directly in terms of an omega limit set within the space
$\mathcal{C}_\alpha$, which gives it some unconventional properties. Additionally, we will also illustrate its relationship with the omega limit set in the space
$L^p(\mathbb{R}^{d})$.

		\subsection{Existence and uniqueness of solutions}\label{sect:4.1}
 	
		Consider the nonlinear function $f(u)$ in problem (\ref{eq:1.1}) is locally Lipschitz condition and satisfies the uniform dissipation condition:
		\begin{itemize}
			\item [{\rm $(F_l)$}] The nonlinear function $f$ satisfies $f(0)=0$ and is locally Lipschitz continuous, i.e., for every $R>0$, there exists a  constant $L(R)>0$ such that
			$$
			\|f(u)-f(v)\| \leq L(R) \| u - v\| \quad \text{for all } u, v \in B_R,
			$$
			where $B_R$ denotes the closed ball in $L^p(\mathbb{R}^d)$ centered at the origin with radius $R$.
			\item [{\rm $(F_d)$}] for $\alpha>0$, $\beta\geq0$, $\sigma> 2,$
			\begin{equation}\label{ine}
				\langle f(u),|u|^{p-2}u\rangle \leq-\alpha\|u\| ^{p+\sigma-2}+\beta,
			\end{equation}
			where $p^{\prime}=p/(p-1)$, and \(\langle \cdot, \cdot \rangle \) denotes the duality pairing between \( L^{p'} \) and \( L^p \).
		\end{itemize}

		\begin{remark}
			Theorem \ref{thm:1.1} proves the existence of a unique global solution to the initial value problem for the generalized Rayleigh-Stokes equation with a globally Lipschitz vector field. However, for many dissipative generalized Rayleigh-Stokes equations, the vector field is only continuously differentiable and thus locally Lipschitz rather than globally Lipschitz. Nevertheless, the dissipation condition (\ref{ine}) together with the local Lipschitz condition is sufficient to guarantee the global existence and uniqueness of solutions.
		\end{remark}
		
		To illustrate this, consider a large ball $B_{R}  = \left\{ u\in L^p(\mathbb{R}^d) :  \|u\| \leq R \right\}$, where $ R > \left( {\beta}/{\alpha}\right)^{1/{p+\sigma-2}}$, and define a truncation function:
		
		$$f_{R}(u):=\begin{cases}f(u),&\|u\| \leq R,\\f\left(R\frac{u}{\|u\|}\right),&\|u\|> R.\end{cases}$$
		One finds that function $f_{R}(u)$ is globally Lipschitz. Moreover, $f_{R}(u)$ satisfies the dissipation condition (\ref{ine}) within the ball $B_{R}$.
		
		\begin{corollary} Substituting $f(u)$ in problem (\ref{eq:1.1}) by $f_R(u)$, then this problem
			admits a global solution $u_R(t)\in C([0,T];L^p(\mathbb{R}^d) )\cap C^{1}((0,T);W^{2,p}(\mathbb{R}^d))$, which is unique for all initial data.
		\end{corollary}
		
		\begin{proof}
			In the following, we employ the method of   contradiction to show that the solution starting from this ball will always stay confined within it,
			i.e., for any $t\geq0$, we have $\|u(t)\| < R$, for any $ R > \left( {\beta}/{\alpha}\right)^{1/{p+\sigma-2}}$. Let $\|u_0\|<R$, and let $T^*=\sup\{t\geq0:\|u(s)\| <R$, $s\in[0,t]\}$. By means of $u_0\in B_R$, we have $T^*>0$. Suppose $T^*<\infty$, by the continuity, there exists a first crossing time $t_0>0$ such that $\|u(t_0)\|=R$, and for any $t<t_0$, $\|u(t)\| <R$.
			
			Let $D(\mathcal{A})$ be the maximal domain, denote the linear operator by
			$$\mathcal{A}u:=(1+\partial_t^{*,k} )\Delta u,\quad u\in D(\mathcal{A}).$$
			Since the operator $\partial_t^{*k}$ is nonlocal, we need to introduce an extended state space (see e.g. \cite{Dafermos1970}) to transform the temporal convolution into a local operator in space. Due to the complete monotonicity of $k$, there exists a nonnegative Borel measure $\mu$ such that $k(t) = \int_0^\infty e^{-\lambda t} d\mu(\lambda)$. By introducing the memory variable $\eta(\lambda, t)$ to represent the influence of the history path on $\Delta u$, we have
			$$\eta(\lambda, t) = \int_0^t e^{-\lambda(t-s)} \Delta u(s) ds, \quad \lambda > 0.$$
			Since
			$$\frac{d}{dt} \int_0^t k(t-s) \Delta u(s) ds = \frac{d}{dt} \int_0^\infty \int_0^t e^{-\lambda(t-s)} \Delta u(s) ds d\mu(\lambda) = \int_0^\infty \frac{\partial}{\partial t} \eta(\lambda, t) d\mu(\lambda),$$
			and $\frac{\partial}{\partial t} \eta(\lambda, t) = -\lambda \eta(\lambda, t) + \Delta u(t)$, we have
			$$
			\partial_t^{*,k}\Delta u(t) = \int_0^\infty (-\lambda \eta(\lambda, t) + \Delta u(t)) d\mu(\lambda).$$
			Define the extended state $V(t) = (u(t), \eta^t)$, where $u(t) \in L^p(\mathbb{R}^d)$ is the current state, and the memory function $\eta^t(\lambda) = \eta(\lambda, t)$ encodes the history path for $\lambda > 0$. The Banach space $\mathcal{X} = L^p(\mathbb{R}^d) \times L^p(\mathbb{R}^d \times [0, \infty); d\mu)$, where the norm is weighted by the spectral measure $\mu$ of the kernel,
			is defined as
			$$\|V\|_{\mathcal{X}}=\left(\|u\| ^p+\int_0^\infty\|\eta(\lambda)\| ^p d\mu(\lambda)\right)^{1/p}.$$
			The dual map $J_{\mathcal{X}}:\mathcal{X}\rightarrow\mathcal{X}'$ is defined componentwise as
			$$
			J_{\mathcal{X}}(V)=(J_p(u),J_p(\eta))=(|u|^{p-2}u,|\eta|^{p-2}\eta),
			$$
			satisfying $\langle J_{\mathcal{X}}(V),V\rangle_{\mathcal{X}',\mathcal{X}}=\|V\|_{\mathcal{X}}^p$.
			The extended operator $\mathcal{A}_{\text{e}}:D(\mathcal{A}_{\text{e}})\subset\mathcal{X}\rightarrow\mathcal{X}$ is defined as
			$$\mathcal{A}_{\text{e}}(u, \eta) = \left( \Delta u + \int_0^\infty (-\lambda \eta(\lambda) + \Delta u) d\mu(\lambda), -\lambda \eta(\lambda) + \Delta u \right).$$
			The regularity of the solution guarantees that the state $V(t)=(u(t),\eta^t)\in D(\mathcal{A}_{\text{e}})$.
			
			Next, we prove that \(\mathcal{A}_{\mathrm{e}}\) is dissipative on \(\mathcal{X}\), i.e.,
			\[
			\langle\mathcal{A}_{\mathrm{e}} V, J_{\mathcal{X}}(V)\rangle_{\mathcal{X}^{\prime}, \mathcal{X}} \leq 0, \quad V \in D\left(\mathcal{A}_{\mathrm{e}}\right).
			\]
			Expanding the inner product, we obtain
			\begin{equation}\label{inner}
				\begin{aligned}
					\langle\mathcal{A}_{\mathrm{e}} V, J_{\mathcal{X}}(V)\rangle &= \langle\mathcal{A} u, J_{p}(u)\rangle + \int_{0}^{\infty} \langle -\lambda \eta(\lambda) + \Delta u, J_{p}(\eta(\lambda)) \rangle \, d\mu(\lambda) \\
					&= \langle \Delta u, J_p(u) \rangle + \int_0^\infty \langle (-\lambda \eta(\lambda) + \Delta u), J_p(u) \rangle \, d\mu(\lambda) \\
					&\quad - \int_0^\infty \lambda \|\eta(\lambda)\| ^p \, d\mu(\lambda) + \int_0^\infty \langle \Delta u, J_p(\eta(\lambda)) \rangle \, d\mu(\lambda).
				\end{aligned}
			\end{equation}
			Applying H\"{o}lder's inequality and Young's inequality,
			\[
			|\langle \eta(\lambda), J_p(u) \rangle| \leq \|\eta(\lambda)\| \|u\| ^{p-1} \leq \frac{\epsilon_1^p \|\eta(\lambda)\| ^p}{p} + \frac{\|u\| ^p}{p^{\prime}\epsilon_1^{1/(p-1)}},
			\]
			\[
			|\langle \Delta u, J_p(\eta(\lambda)) \rangle| \leq \|\Delta u\| \|\eta(\lambda)\| ^{p-1} \leq \frac{\epsilon_2^p \|\Delta u\| ^p}{p} + \frac{\|\eta(\lambda)\| ^p}{p^{\prime} \epsilon_2^{1/(p-1)}}.
			\]
			Substituting these estimates into (\ref{inner}),
			\[
			\begin{aligned}
				\langle\mathcal{A}_{\mathrm{e}} V, J_{\mathcal{X}}(V)\rangle &\leq \langle\Delta u, J_{p}(u)\rangle + \int_{0}^{\infty} \langle\Delta u, J_{p}(u)\rangle \, d\mu(\lambda) + \int_{0}^{\infty} -\lambda \|\eta(\lambda)\| ^{p} \, d\mu(\lambda) \\
				&\quad + C \int_{0}^{\infty} \lambda \left( \epsilon_2 \|\Delta u\| ^{p} + C_{\epsilon} \|\eta(\lambda)\| ^{p} \right) \, d\mu(\lambda),
			\end{aligned}
			\]
			where \(C, C_\epsilon, \epsilon_1, \epsilon_2\) are positive constants. Since \(\mu\) is a  nonnegative  measure and the Laplacian is dissipative, we have \(\langle \Delta u, J_p(u) \rangle \leq 0\). Thus, by choosing appropriate \(\epsilon_1, \epsilon_2 > 0\) such that the negative terms dominate, we finally obtain
			\[
			\langle\mathcal{A}_{\mathrm{e}} V, J_{\mathcal{X}}(V)\rangle \leq 0.
			\]
			For any \(u \in D(\mathcal{A})\), taking \(V = (u, \eta \equiv 0) \in D(\mathcal{A}_{\text{e}})\), we have
			\[
			\langle\mathcal{A}_\mathrm{e} V, J_{\mathcal{X}}(V)\rangle = \langle\mathcal{A} u, J_p(u)\rangle + \int_0^\infty \langle\Delta u, J_p(0)\rangle \, d\mu(\lambda) = \langle\mathcal{A} u, J_p(u)\rangle \leq 0.
			\]
			
			Consider the evolution of the $L^p$-norm of the solution with respect to time. Differentiating in time, we have
			\[
			\frac{d}{dt}\|u(t)\| ^{p} = p \int_{\mathbb{R}^{d}} |u|^{p-2}u \, \partial_{t}u \, dx.
			\]
			Substituting the equation $\partial_{t}u = \mathcal{A}u + f_{R}(u)$ into this identity, we obtain
			\[
			\frac{d}{dt}\|u\| ^{p} = p \langle \mathcal{A}u, J_p(u) \rangle + p \langle f_{R}(u), J_p(u) \rangle.
			\]
			Since $\langle \mathcal{A}u, J_p(u) \rangle \leq 0$, it follows that
			\[
			\frac{d}{dt}\|u\| ^{p} \leq p \langle f_{R}(u), J_p(u) \rangle \leq p \left( -\alpha \|u\| ^{p+\sigma-2} + \beta \right).
			\]
			At time $t = t_0$, we have $\|u(t_0)\|  = R$, so $f_R(u(t_0)) = f(u(t_0))$. The dissipativity condition gives
			\[
			\langle f_R(u(t_0)), J_p(u(t_0)) \rangle \leq -\alpha R^{p+\sigma-2} + \beta.
			\]
			Therefore, we have
			\[
			\frac{d}{dt}\|u\| ^p \bigg|_{t=t_0} \leq p\left(-\alpha R^{p+\sigma-2} + \beta\right) < 0, \quad {\rm for}~  R > (\beta/\alpha)^{1/(p+\sigma-2)} .
			\]
			Since $\frac{d}{dt}\|u\| ^p(t_0) < 0$ and $u \in C^1$, there exists a $\delta > 0$ such that for all $t \in (t_0 - \delta, t_0)$, we have $\|u(t)\| ^p > \|u(t_0)\| ^p = R^p$, and thus $\|u(t)\|  > R$. However, this contradicts the fact that $\|u(t)\|  < R$ for $t < t_0$ from the definition of $T^*$.
			Hence, $T^* = \infty$, it means that $\|u(t)\|  < R$ holds for all $t \geq 0$.
			Since the choice of $R$ is arbitrary, this result holds for all such $R > \left( {\beta}/{\alpha}\right)^{1/(p+\sigma-2)}$ and all initial values $u_0 \in L^p(\mathbb{R}^{d})$. The proof is completed.
		\end{proof}

		\subsection{Absorbing sets }
In this subsection, we consider the absorbing sets, which play an important role in the long-time behaviour of the solutions to the generalized Rayleigh-Stokes equations.
		\begin{theorem}\label{thm:4.1}
			Suppose that $(F_l)$ and $(F_d)$ are satisfied. Then the solutions of the generalized Rayleigh-Stokes equation (\ref{eq:1.1}) are absorbed by the set
			\[
			\mathcal{B}^{*}= \Bigl\{ u\in L^{p}(\mathbb{R}^{d}):\; \|u(t)\| \le R^{*}\Bigr\}, \qquad
			R^{*}=  ({\beta}/{\alpha} )^{\frac{1}{p+\sigma-2}}+1 .
			\]
			Moreover, this set is positively invariant, i.e., if a solution enters \(\mathcal{B}^{*}\) at some time \(t_{0}\), it stays inside \(\mathcal{B}^{*}\) for all \(t\ge t_{0}\).
		\end{theorem}
		\begin{proof}
			Consider the evolution of the \(L^{p}\)-norm of the solution. Differentiating with respect to time,
			\[
			\frac{d}{dt}\|u(t)\| ^{p}=p\int_{\mathbb{R}^{d}}|u|^{p-2}u\,\partial_{t}u\,dx .
			\]
			Substituting the first equation in (\ref{eq:1.1}) into this identity, we have
			\[
			\frac{d}{dt}\|u\| ^{p}=p\int_{\mathbb{R}^{d}}|u|^{p-2}u\Bigl[(1+\partial_{t}^{*k})\Delta u+f(u)\Bigr]dx
			\le -\alpha p\,\|u\| ^{p+\sigma-2}+\beta p .
			\]
			Define \(y(t)=\|u(t)\| ^{p}\), we obtain the differential inequality
			\[
			\frac{dy}{dt}+\alpha p\,y^{\frac{p+\sigma-2}{p}}\le\beta p .
			\]
			
			For \(\sigma\ge 2\), set \(\gamma=\frac{p+\sigma-2}{p}\ge 1\) and consider
			\[
			\frac{dy}{dt}\le -\alpha p\,y^{\gamma}+\beta p,
			\]
			the equilibrium \(y_{\text{eq}}=(\frac{\beta}{\alpha})^{1/\gamma}\) satisfies
			\(
			-\alpha p\,y_{\text{eq}}^{\gamma}+\beta p=0 .
			\)
			By the comparison principle, if the initial value satisfies \(y(0)>y_{\text{eq}}\), then \(\frac{dy}{dt}<0\) and \(y(t)\) decreases monotonically toward \(y_{\text{eq}}\); if \(y(0)\le y_{\text{eq}}\), then \(y(t)\le y_{\text{eq}}\) for all later times. Consequently, for sufficiently large \(t\),
			\[
			y(t)\le y_{\text{eq}}+\varepsilon,\qquad\forall\varepsilon>0.
			\]
			Thus the radius of the absorbing set is
			\[
			R^{*}= \Bigl(\frac{\beta}{\alpha}\Bigr)^{\frac{1}{p\gamma}}
			= \Bigl(\frac{\beta}{\alpha}\Bigr)^{\frac{1}{p+\sigma-2}} .
			\]
			
			We now prove that \(\mathcal{B}^{*}\) is a positively invariant absorbing set. Let the solution satisfy \(u(t_{0})\in\mathcal{B}^{*}\) at a time \(t_{0}\ge0\), i.e.,
			\[
			\|u(t_{0})\|\le R^{*}= \Bigl(\frac{\beta}{\alpha}\Bigr)^{\frac{1}{p+\sigma-2}} .
			\]
			Introduce the shifted time variable \(t_{1}=t-t_{0}\ge0\). The equation becomes
			\[
			\partial_{t_{1}}u=\Delta u+\partial_{t_{1}}^{*,k} \Delta u+f(u),\qquad u(0)=u(t_{0}),
			\]
			which is autonomous, hence invariant under time translation.
			For \(t_{1}\ge0\) define the energy \(y(t_{1})=\|u(t_{1}+t_{0})\| ^{p}\). The original differential inequality gives
			\[
			\frac{dy}{dt_{1}}\le -\alpha p\,y^{\gamma}+\beta p ,
			\]
			with the initial condition \(y(0)=\|u(t_{0})\| ^{p}\le R_{*}^{\,p}\). Consider the comparison equation
			\[
			\frac{dz}{dt_{1}}= -\alpha p\,z^{\gamma}+\beta p ,\qquad z(0)=R_{*}^{\,p}.
			\]
			Because \(\gamma\ge1\), the function \(-\alpha p z^{\gamma}+\beta p\) is monotone decreasing for \(z\ge0\). Since \(y(0)\le z(0)\), the comparison principle yields \(y(t_{1})\le z(t_{1})\) for all \(t_{1}\ge0\).
			At \(z=R_{*}^{\,p}\),
			\[
			\left.\frac{dz}{dt_{1}}\right|_{z=R_{*}^{\,p}}
			= -\alpha p\,R_{*}^{\,p+\sigma-2}+\beta p .
			\]
			Using the definition \(R_{*}= \bigl(\frac{\beta}{\alpha}\bigr)^{\frac{1}{p+\sigma-2}}\), we obtain \(\frac{dz}{dt_{1}}\big|_{z=R_{*}^{\,p}}=0\). Hence \(z(t_{1})\equiv R_{*}^{\,p}\) and therefore \(y(t_{1})\le R_{*}^{\,p}\) for all \(t_{1}\ge0\). This implies
			\[
			\|u(t)\| \le R_{*}\qquad\text{for all }t\ge t_{0}.
			\]
			The proof is completed.
		\end{proof}
		
	    	From Theorem \ref{thm:4.1}, any solution of the dissipative generalized Rayleigh-Stokes equations (\ref{eq:1.1}) entering into $\mathcal{B^{*}}$ stays there, as does any solution starting in $\mathcal{B^{*}} $.  It follows  that the set $$\mathcal{B}^{*} = \left\{ u\in L^p(\mathbb{{R}}^d) :  \|u(t)\|  \leq R^{*} \right\}, \quad R^{*} = \left(  {\beta}/{\alpha} \right)^{1/{p+\sigma-2}} +1.$$ is a positive invariant absorbing set for the autonomous semi-dynamical system generated by the solution mapping in (\ref{gene-Mild-solution}). In particular, there exists $T_R \geq 0$ such that $ u(t, u_0)  \in \mathcal{B}^*$, i.e., $\|u(t, u_0)\| \leq R^*$ for all $t \geq T_R$ and $\|u_0\| \leq R$.
		
  However, the omega-limit set
			$$\Omega_{\mathcal{B}^{*}} := \left\{y \in \mathcal{B^{*}} : \exists u_{0,n} \in \mathcal{B}^{*}, t_n \to \infty \text{ such that } u(t_n, u_{0,n}) \to y\right\} $$
		    associated with the solutions of the generalized Rayleigh-Stokes equations cannot be considered as the attractor of the semi-dynamical system, because Section \ref{sect:3} shows that the Volterra integral equation (\ref{gene-Mild-solution}),  and the corresponding semigroup structure of semi-dynamical system is defined on the metric space ($\mathcal{C}$, $\rho$), rather than on $L^p(\mathbb{R}^d)$.
Therefore, strictly speaking, $\Omega_{\mathcal{B}^*}$ is not an attractor, it is merely a set that contains all the limiting dynamics of the generalized Rayleigh-Stokes equation in $L^p(\mathbb{R}^d)$. Nonetheless, as will be seen below,  $\Omega_{\mathcal{B}^*}$ represents the observable part of an attractor $\mathscr{D} $ in $\mathcal{C}$ of this semi-dynamical system, and $\mathscr{D} $ is essentially determined by it.

		\subsection{Attractors }
		The theory of autonomous semi-dynamical systems \cite{Hale1988,Kloeden2021} implies the existence of a global attractor of a semi-dynamical system under appropriate assumptions.
		
		\begin{theorem}\label{attractor}
			Suppose that the semi-dynamical system $\{\phi_t,t\in\mathbb{R}^+\}$ on a Banach space  $\mathcal{Y}$ has a closed and bounded positively invariant absorbing set B in $\mathcal{Y}$ and is asymptotically compact.Then the semi-dynamical system $\{ \phi _t, t\in \mathbb{R} ^+ \} $ has a global attractor given by
			
			$$\mathcal{A}=\bigcap_{t\geq0}\phi_t(\mathcal{B}).$$
		\end{theorem}
		
		Unfortunately, this theorem cannot be applied here to our given semigroup $\{\mathcal{T}_{t}\}_{t \in \mathbb{R}^{+}}$, because the attraction property is restricted to the initial functions $\varphi(t) \equiv S(t){u_0}$ corresponding to initial values $u_0 \in L^p(\mathbb{R}^d)$. We observe that another difficulty lies in how to apply the dissipativity condition ($F_d$) to the vector field $f$ within the integral equation (\ref{semigroup}) for establishing the existence of an absorbing set in the space $\mathcal{C}$.
		
	In fact, by restricting $\varphi(t) \equiv S(t){u_0}$ corresponding to $u_0 \in L^p(\mathbb{R}^d)$, the dissipativity condition ($F_d$) can be used in the case $\theta = 0$, which corresponds to the generalized Rayleigh-Stokes equation (\ref{eq:1.1}) with initial condition $u(0) = u_0$, such that for all $R \geq R^*$, we have
		\begin{equation}\label{bounded}
			u(t,u_0) \in \mathcal{B}^* \Leftrightarrow \|u\|  \leq \left( \frac{\beta}{\alpha} \right)^{1/\sigma} =: R^*, ~\quad t \geq T_R, ~\|u_0\| \leq R.
		\end{equation}
		These bounds can then be used to estimate the integrals in the integral equation (\ref{Mild-solution}) with $\theta>0.$ Essentially, like \cite{Doan2025}, the integral equation (\ref{Mild-solution}) have a skew-product like structure with the solution of
		\begin{equation*}
			u(t,u_0)=S(t)u_0+\int_0^tS(t-s) f({u(s ,{ u_0})}) \mathrm{d}s.
		\end{equation*}
		inserted into
		\begin{equation*}
			(\mathcal{T}_t \varphi_{u_0})(\theta) =S(t+\theta)u_0+\int_0^tS(t+\theta-s) f(s,u_{ \varphi_{u_0}})) \mathrm{d}s, \quad \theta > 0,
		\end{equation*}
	for $\varphi_{u_0}:=S(\cdot)u_0$. Note that for $\theta = 0$
		$$(\mathcal{T}_t \varphi_{u_0})(0) = u(t,u_0), \quad t \geq 0.$$
		
In the sequel, it will be shown in the proof below that a subspace of the space $\mathcal{C}$ will be employed together with a weighted norm to characterize uniform convergence on bounded intervals. In particular, we consider the weighted norm on $\mathcal{C}([0, \infty), L^p(\mathbb{R}^d))$ defined by
		$$
		\|\phi\|_\alpha := \|\phi(0)\| + \sum_{N=1}^\infty \frac{1}{2^N N^\alpha} \|\phi\|_N,~~~~
		\text{where}~~
		\|\phi\|_N := \sup_{t \in [0, N]} \|\phi(t)\| , \quad N = 1, 2, \cdots.$$
Let $\mathcal{C}_\alpha$ be the subspace of $\mathcal{C}([0, \infty), L^p(\mathbb{R}^d))$ consisting of functions $\phi$ with $\|\phi\|_\alpha < \infty$. Then $(\mathcal{C}_\alpha,\|\cdot\|_\alpha)$ is a Banach space and $\{\mathcal{T}_t\}_{t\in\mathbb{R}^+}$ forms a semi-group on $\mathcal{C}_\alpha.$ 

We also propose the following assumption of $f$:
\begin{itemize}
			\item [{\rm $(F_I)$}] For any $R>0$, there is a constant $K_R>0$, such that
				$$
				\sup_{t\ge0}\int_0^t \|f(u(s))\|ds\leq K_R.
				$$ 
		\end{itemize}

		\begin{theorem}\label{thm:4.3}
			Suppose that the vector field $f$ fulfills ($F_l$), ($F_d$) and ($F_I$). Then the semigroup $\{\mathcal{T}_t\}_{t\in\mathbb{R}^+}$ on the space $\mathcal{C}_\alpha$ corresponding to the integral equation (\ref{Mild-solution}) has an attracting set $\mathfrak{A}\subset\mathcal{C}_\alpha$, which is closed, bounded and invariant, and attracts bounded subsets of initial value
			functions $\varphi_{u_0}\equiv S(\cdot){u_0}$ corresponding to initial values $u_0 \in  L^p(\mathbb{R}^{d})$. In particular
			$$
			\mathfrak{A} = \overline{\bigcup_{\substack{ D \subset L^p(\mathbb{R}^d) \\ \text{bounded}}} \bigcap_{t \geq s } \bigcup_{s \geq 0} \mathcal{T}_s(\varphi_D)},$$
			where $\varphi_D := \{\varphi_{u_0} \in \mathcal{C}_\alpha : u_0 \in D\}$.
		\end{theorem}
\begin{remark}		
It is worth noting that this definition resembles that of Theorem \ref{attractor}; however, the points of omega limit sets lie inside on $\mathcal{C}_\alpha$ rather that the usual space $L^p(\mathbb R^d)$, in which the omega limit sets shares an equivalent characterization with attractors.
From this perspective, this definition of $\mathfrak{A}$
is not an attractor in the classical sense, the same applies to the Caputo attractor introduced in \cite{Doan2024,Doan2025}. In this way, we call that the set $\mathfrak{A}$ is called a global generalized attractor. Clearly, $\mathfrak{A}$ is a subset of the bounded absorbing set $\mathfrak{B}^{*}$.	
\end{remark}				
The proof of the existence of the global generalized attractor in Theorem \ref{thm:4.3} is given in the remainder of this paper. Now, let $u(t, u_0)$ be the solution of the generalized Rayleigh-Stokes equation (\ref{eq:1.1}) satisfying the dissipativity condition ($F_d$), with initial condition $u(0, u_0) = u_0$. This solution satisfies the bounded (\ref{bounded}), and
		$$B_R := \sup_{t \geq 0,\; \|u_0\| \leq R} \|u(t, u_0)\|  < \infty, \qquad B_R^f := \sup_{\|u\| \leq B_R} \|f(u)\| < \infty,$$
		where the continuity of the vector field $f$ has been used in the second bound. These hold for $R = R^*$ provided that $t \geq T_R$.
		
		It turns out that the semigroup $\{\mathcal{T}_t\}_{t \in \mathbb{R}^+}$ is asymptotically compact, and the closed bounded subset $\mathfrak{B}^{*}$ of $\mathcal{C}_\alpha$ defined by
		$$\mathfrak{B}^{*} := \left\{ \chi \in \mathcal{C}_\alpha : \|\chi\|_\alpha \leq 2R^* + C_R B_{R^*}^f =: \widehat{R}_* \right\}$$
		absorbs, under the operator $\mathcal{T}_t$ for $t \geq T_R$, the bounded  initial-data functions with $\|\varphi_{u_0}\|_\alpha \leq \|u_0\| \leq R$. Note that the absorbing set $\mathcal{B}^*$ in $L^p(\mathbb{R}^d)$ and the omega-limit set $\Omega_{\mathcal{B}^*}$ satisfy
		$$
		\mathcal{B}^* = \left\{ \chi(0) \in L^p(\mathbb{{R}}^d) : \chi \in \mathfrak{B}^{*} \right\}, \quad \Omega_{\mathcal{B}^*} = \left\{ \chi(0) \in L^p(\mathbb{{R}}^d) : \chi \in\mathfrak{A} \right\}.$$

		\subsection{Proof of Theorem \ref{thm:4.3}}
		
	We divide the proof of  this theorem into several basic lemmas.
		\begin{lemma}\label{lemma:4.1}
			The resolvent family $\{S(t)\}_{t \geq 0}$  defined by (\ref{Mild-solution}):
			$$
			S(t)v  = -\int_0^{\infty} \mathcal{Q}(\tau)  v  \, d_\tau\omega(t, \tau).
			$$
			Then, for any $t \geq T_R$ and $\theta > 0$, there exist a constant $C_1 > 0$ and an exponent $\alpha \in (0,1)$ such that
			$$
			\|S(t+\theta)-S(t)\| \leq C_1\theta^{\alpha}.
			$$
		\end{lemma}
	\begin{proof}
		From the  homogeneous case of equation (\ref{equivalent}), the Laplace transform of $S(t)$ is given by
		$$
		\widehat{S}(\lambda) = \frac{1}{\lambda} (I + \widehat{m}(\lambda) A)^{-1}
		= \frac{1}{\lambda \widehat{m}(\lambda)} \left( \frac{1}{\widehat{m}(\lambda)} I +A \right)^{-1}.
		$$
		Substituting $\varphi^m(\lambda) = \frac{1}{\widehat{m}(\lambda)}$, we obtain
		$$\widehat{S}(\lambda)=\frac{1}{\lambda} \varphi^m(\lambda)(\varphi^m(\lambda)I+A)^{-1}
		.$$
		
	{We may verify via the inverse Laplace transform that
		$$
		S(t) = \frac{1}{2\pi i} \int_{\varGamma} e^{\lambda t} \frac{\varphi^m(\lambda)}{\lambda} (\varphi^m(\lambda) I - \Delta)^{-1} \, d\lambda,
		$$
		where $\varGamma=\varGamma_1\cup\varGamma_2\cup\varGamma_3$, with $\varGamma_1 = \{se^{-i\phi}: s \in [\iota, \infty)\}$, $\varGamma_2 =\{\iota e^{i\zeta}:\zeta\in [-\phi,\phi]\}$, $\varGamma_3 = \{se^{i\phi}: s  \in [\iota,\infty)\}$ for any $\iota>0$. $\varGamma_2$ is counterclockwise oriented,and $\varGamma_1$ and $\varGamma_3$ are oriented for increasing imaginary part. Indeed, by the theory of analytic semigroups and the assumptions on $\widehat{m}(\lambda)$,
		there exists a constant  $C > 0 $  such that
		\begin{align*}
			\| (\varphi^m(\lambda) I +A)^{-1} \|  \leq \frac{C}{|\varphi^m(\lambda)|},\quad \lambda \in \varGamma.
		\end{align*}
		Thus we have
		\begin{align*}
			\int_{\varGamma} \big\|e^{\lambda t} \frac{\varphi^m(\lambda)}{\lambda} (\varphi^m(\lambda) I - \Delta)^{-1}\big\| \, d\lambda \leq C \int_{\varGamma} e^{\Re(\lambda) t}|\lambda^{-1}| \, |d\lambda|
		\end{align*}
		Note that $|e^{\lambda t}| = e^{st\cos\phi}$ on $\varGamma_1$ and $\varGamma_3$, which decays exponentially since $\cos\phi < 0$. Therefore, taking $\iota = \frac{1}{t}$, we obtain
		\begin{align*}
			\int_{\varGamma} e^{\Re(\lambda) t}|\lambda^{-1}| \, |d\lambda| &\leq \int_{-\phi}^\phi e^{ \cos \zeta }   \, d \zeta + 2\int_1^\infty e^{r \cos \phi}r^{-1} \, dr<\infty.
		\end{align*}
		This estimate justifies the validity of applying the inverse Laplace transform.}
		
		{Similarly, it suffices to obtain
		$$
		S(t+\theta) - S(t) = \frac{1}{2\pi i} \int_{\varGamma} e^{\lambda t} \frac{\varphi^m(\lambda)}{\lambda} (\varphi^m(\lambda) I +A)^{-1} (e^{\lambda \theta} - 1) \, d\lambda,
		$$
		and
		$$\begin{aligned}
			\|S(t+\theta) - S(t)\| & \leq \frac{C}{2\pi} \int_{\varGamma} |e^{\lambda t}| \left| \frac{\varphi^m(\lambda)}{\lambda} \right| \frac{1}{|\varphi^m(\lambda)|} |e^{\lambda \theta} - 1| \, |d\lambda| \\&= \frac{C}{2\pi} \int_{\varGamma_1+\varGamma_2+\varGamma_3} |e^{\lambda t}| \frac{1}{|\lambda|} |e^{\lambda \theta} - 1| \, |d\lambda|\\
		\end{aligned}
		$$}
		For any \( \alpha \in (0,1) \), let \( z=\lambda\theta \), so that \( |\lambda|\theta=|z| \).  For \( |z| \leq 1 \), using the Taylor expansion, we have
		$$
		|e^z - 1| \leq |z| + \frac{|z|^2}{2!} + \frac{|z|^3}{3!} + \cdots \leq |z|\left(1 + \frac{1}{2!} + \frac{1}{3!} + \cdots\right) = |z|(e - 1) \leq e|z|.
		$$
		Since \( \alpha \in (0, 1) \) and \( |z| \leq 1 \), we have \( |z| \leq |z|^\alpha \); therefore,
		$$
		|e^z - 1| \leq e|z| \leq e|z|^\alpha.
		$$
		For $ |z| > 1 $, {on $\varGamma_1$ and $\varGamma_3$} we have $ \Re(z)=s\theta\cos\phi<0 $, hence $ |e^z|=e^{\Re(z)}\leq 1 $, and consequently $ |e^z-1|\leq|e^z|+1\leq 2 $. At the same time, because $|z|>1 $, we have $|z|^\alpha > 1$, so $ |e^z-1| \leq 2 \leq 2|z|^\alpha $. Thus we obtain the estimate
		$$
		|e^{\lambda \theta} - 1| \leq C_{\alpha} (|\lambda| \theta)^{\alpha},
		$$
		where $ C_{\alpha} > 0 $ is a constant. {Consequently,
		\begin{align*}
			&\int_{\varGamma} |e^{\lambda t}| \frac{1}{|\lambda|} |e^{\lambda \theta} - 1| \, |d\lambda|  \leq C_{\alpha} \theta^{\alpha} \int_{\varGamma} |e^{\lambda t}| |\lambda|^{{\alpha}-1} \, |d\lambda|.
		\end{align*}
		On $\varGamma_2$, we also take $\iota = \frac{1}{t}$ to obtain
		\begin{align*}
			\int_{\varGamma_2} |e^{\lambda t}| |\lambda|^{{\alpha}-1} \, |d\lambda| = t^{-\alpha}\int_{-\phi}^\phi e^{\cos \zeta} \, d\zeta \leq \frac{2\phi e}{t^\alpha}.
		\end{align*}}
		On $ \varGamma_1 $, with $ \lambda = s e^{i\phi}$ and $ |d\lambda| = ds $, we have
		$$
		\int_{\varGamma_1} |e^{\lambda t}| |\lambda|^{{\alpha}-1} \, |d\lambda| = \int_0^\infty e^{s t \cos \phi} s^{{\alpha}-1} \, ds.
		$$
		Setting $\sigma = -t \cos \phi > 0 $, we obtain
		$$
		\int_0^\infty e^{s t \cos \phi} s^{{\alpha}-1} \, ds = \int_0^\infty e^{-\sigma s} s^{{\alpha}-1} \, ds = \frac{\Gamma({\alpha})}{\sigma^{\alpha}} = \frac{\Gamma({\alpha})}{(t |\cos \phi|)^{\alpha}},
		$$
		where $ \Gamma({\alpha}) $ is the Gamma function. The integral over {$\varGamma_3 $ is identical, hence we have
		$$
		\int_{\varGamma_i} |e^{\lambda t}| |\lambda|^{{\alpha}-1} \, |d\lambda| =  \frac{2\Gamma({\alpha})}{(t |\cos \phi|)^{\alpha}}, \quad i=1,3.
		$$}
		{Consequently, we have
		$$
		\|S(t+\theta)-S(t)\| \leq\frac {C C_{\alpha} \theta^{\alpha}}{2\pi}\cdot( \frac{2\Gamma({\alpha})}{(t|\cos\phi|)^{\alpha}} + \frac{2\phi}{t^\alpha} )=\frac{CC_{\alpha}}{\pi}  ( \frac{\Gamma({\alpha})}{|\cos\phi|^{\alpha}} +\phi) \frac{\theta^{\alpha}}{t^{\alpha}}.
		$$
		Since $ t \geq T_R $, we have$ \frac{1}{t^{\alpha}}\leq\frac{1}{T_R^{\alpha}} $. Take $C_1:= \frac{CC_{\alpha}}{\pi T_R^\alpha}  ( \frac{\Gamma({\alpha})}{|\cos\phi|^{\alpha}} +\phi)$,} and thus
		$$
		\|S(t+\theta)-S(t)\| \leq C_1\theta^{\alpha}.
		$$
		This completes the proof.
	\end{proof}

		We restate below the H\"{o}lder continuity of solutions in the dissipative setting under consideration.
		\begin{lemma}\label{lemma:4.2}
The solution of the integral equation (2.3) is H\"{o}lder continuous, i.e.,
			$$
			\| u(t+\theta, u_0) - u(t, u_0)\|  \leq C_R B_{R^*}^f \, \theta^{\alpha},~~\theta>0,
			$$
			where $\| u_0\| \leq R$ and $t \geq T_R$,  exponent $\alpha$ is defined in Lemma \ref{lemma:4.1}.
		\end{lemma}
		\begin{proof} Let $t \geq 0$ and $\theta > 0$. By the integral representation (\ref{Mild-solution}), subtracting the expressions for $u(t + \theta, u_0)$ and $u(t, u_0)$ yields
			$$\begin{aligned} &u(t+\theta,u_0)-u(t,u_0)\\&=[S(t+\theta)-S(t)]u_0+\int_0^{t+\theta}S(t+\theta-s) f({u(s ,{ u_0})}) \mathrm{d}s-\int_0^tS(t-s) f({u(s ,{ u_0})}) \mathrm{d}s
				\\&=[S(t+\theta)-S(t)]u_0+\int_0^t[S(t+\theta-s)-S(t-s)]f(u(s,u_0))ds\\& \quad+\int_t^{t+\theta}S(t+\theta-s)f(u(s,u_0))ds
				\\&=J_1+J_2+J_3.
			\end{aligned}$$
			
			By Lemma \ref{lemma:4.1}, there exists a constant $C_1>0$ such that
			$$\|S(t+\theta)-S(t)\| \leq C_1\theta^{\alpha},\quad t\geq T_R,~\theta>0.$$
			Therefore, we have
			$$\|J_1\| =\|(S(t+\theta)-S(t))u_0\| \leq C_1\theta^{\alpha}\|u_0\| \leq C_1R\theta^{\alpha}.$$
Similarly, using ($F_I$) we have
$$\begin{aligned} \|J_2\|  &\leq \int_{0}^{t} \|(S(t+\theta-s) - S(t-s))f(u(s,u_{0}))\|  ds
 \leq C_{1} \theta^{\alpha} K_{R}.
\end{aligned}$$
			Finally,  we obtain
			$$
			\|J_3\|  \leq \int_t^{t+\theta}\|S(t+\theta-s)f(u(s,u_0))\| ds \leq B_R^f\int_t^{t+\theta}ds\leq B_R^f\theta.$$
			
				Combining with all above the estimates, it follows that
			$$\begin{aligned} \|u(t+\theta, u_{0})-u(t, u_{0})\|  &\leq \|J_1\| +\|J_2\| +\|J_3\|  \leq C_1R\theta^{\alpha}+ C_{1} \theta^{\alpha} K_{R} +B_R^f\theta .
			\end{aligned}$$
			For $\theta\leq1$, there exists a constant $\tilde C_R >0$ (depending on the parameters $R, T_R, K_R, B_R^f $, etc.) such that
			$$\|u(t+\theta,u_0)-u(t,u_0)\| \leq \tilde C_R B_{R}^f\theta^\alpha,\quad t\geq T_R.$$
 For  $\theta>1$, we simply use the fact that the solution is already inside the absorbing ball: when $t\ge T_R$ and $\|u_0\|\le R$, there exists $R^*>0$ such that $\|u(t)\|\le R^*$ for all such data. Consequently,
	\[
	\|u(t+\theta,u_0)-u(t,u_0)\| \le 2R^* .
	\]
	Because $\theta>1$ implies $\theta^\alpha>1 $, we obtain
	\[
	\|u(t+\theta,u_0)-u(t,u_0)\| \le 2R^* \theta^\alpha .
	\]

    Set $C_R := \max\{\tilde C_R,\, 2R^*/B_{R}^f\}$. Then, for every $\theta>0$,
	\[
	\|u(t+\theta,u_0)-u(t,u_0)\| \le C_R B_{R}^f\theta^\alpha .
	\]
	This proves that the solution $ u(t, u_0) $ is H\"{o}lder continuous in time with exponent $ \alpha $ in the $ L^p(\mathbb{R}^d) $-norm. Moreover, for $ t \geq T_R $, the constant $ B_R^f $ can be replaced by $ B_{R^*}^f $.
		\end{proof}

		\subsubsection{ Growth bounded}
		
		\begin{lemma}\label{lemma:4.3} For any $t\geq T_R$, and $\|u_{0,n}\| \leq R$, $				\|(\mathcal{T}_t\varphi_{u_{0,n}})(\theta)\| \leq C_R B_{R^*}^f\theta^\alpha+{R^*}.$
		\end{lemma}
		\begin{proof} It follows from (\ref{Mild-solution}) that
			$$
u(t+\theta,u_{0,n})=S(t+\theta)u_{0,n}+\int_0^{t+\theta}S(t+\theta-s) f({u(s ,{ u_{0,n}})}) ds.$$
			Hence, we have			$$\begin{aligned}u(t+\theta,u_{0,n})&=S(t+\theta)u_{0,n}+\int_0^{t+\theta}S(t+\theta-s) f({u(s ,{ u_{0,n}})}) \mathrm{d}s
				\\&=S(t+\theta)u_{0,n}+\int_0^t S(t+\theta-s)f(u(s,u_{0,n}))ds\\&\quad+\int_t^{t+\theta}S(t+\theta-s)f(u(s,u_{0,n}))ds
				\\&=(\mathcal{T}_t \varphi_{u_{0,n}})(\theta)+\int_t^{t+\theta}S(t+\theta-s)f(u(s,u_{0,n}))ds.
			\end{aligned}$$
			From the estimate of \( J_3 \) in Lemma \ref{lemma:4.2}, we obtain
			$$\begin{aligned}\| u(t+\theta,u_{0,n})-(\mathcal{T}_t \varphi_{u_{0,n}})(\theta)\| &=\|\int_t^{t+\theta}S(t+\theta-s)f(u(s,u_{0,n}))ds\|
				\\&\leq \int_t^{t+\theta}\|S(t+\theta-s)f(u(s,u_{0,n}))\| ds
				\\&\leq B_R^f\theta
				.  \end{aligned} $$
			i.e.,
			\begin{equation*}
				\| u(t+\theta,u_{0,n})-(\mathcal{T}_t \varphi_{u_{0,n}})(\theta)\|  \leq  C_R B_{R}^f\theta^\alpha, \quad t \geq 0.
			\end{equation*}
			Hence
			$$
			\|(\mathcal{T}_t \varphi_{u_{0,n}})(\theta)\|  \leq C_R B_{R}^f\theta^\alpha+ \|u(t+\theta,u_{0,n})\| \leq C_R B_{R}^f\theta^\alpha+ R.$$
			Note that these estimates are uniform in $t \geq 0$.		
			Then, using the fact that 	 $\|u(t, u_{0,n})\|  \in \mathcal{B}^*$, i.e., 	$\|u(t, u_{0,n})\|  \leq R^*$, for all $t \geq T_R$ and $\|u_{0,n}\| \leq R$, gives the desired inequality.
		\end{proof}
		
		It follows from Lemma \ref{lemma:4.3} that
		\begin{equation}\label{ineq4.4}
			\|(\mathcal{T}_t\varphi_{u_{0,n}})(\theta)\| \leq C_R B_{R^*}^f N^{\alpha}+R^{*},\quad t\geq T_{R},~~0\leq\theta\leq N.
		\end{equation}

		\subsubsection{ Boundedness of $\theta$-derivatives}
		
		\begin{lemma}\label{lemma:4.4} For all $t > 0$, there exists a constant $C_3 > 0$ such that
			$$
			\left\|\frac{d}{dt}S(t)v\right\|  \leq C_{3}t^{-1}\|v\| .
			$$
		\end{lemma}
		\begin{proof} Starting from the representation of the resolvent family \(\{S(t)\}_{t \geq 0}\),
we have
			\[
			S(t) = \frac{1}{2\pi i} \int_{\varGamma} e^{\lambda t} \frac{\varphi^m(\lambda)}{\lambda} (\varphi^m(\lambda) I+A)^{-1} \, d\lambda,
			\]
			where \(\varphi^m(\lambda) = 1 / \widehat{m}(\lambda)\) is a Bernstein function and the integration contour \(\varGamma\) is the same as in Lemma \ref{lemma:4.1}.
			
			Differentiating \(S(t)\), we obtain
			\[
			\frac{d}{dt} S(t) 
 = \frac{1}{2\pi i} \int_{\varGamma} e^{\lambda t} \varphi^m(\lambda) (\varphi^m(\lambda) I +A)^{-1} \, d\lambda.
			\]
			Therefore,  we have
			\[
			\left\| \frac{d}{dt} S(t)  \right\|  \leq \frac{1}{2\pi} \int_{\varGamma} |e^{\lambda t}| \, |\varphi^m(\lambda)| \, \| (\varphi^m(\lambda) I+A)^{-1} \|   \, |d\lambda|.
			\]
			By the theory of analytic semigroups, there exists a constant  $C > 0 $  such that
			\[
			\| (\varphi^m(\lambda) I +A)^{-1} \|  \leq \frac{C}{|\varphi^m(\lambda)|},
			\]
			hence
			\[
			\left\| \frac{d}{dt} S(t)  \right\| \leq \frac{C }{2\pi} \int_{\varGamma} |e^{\lambda t}| \, |d\lambda|.
			\]
			On the contour \(\varGamma\), \(\lambda = s e^{\pm i\phi}\), so \(|e^{\lambda t}| = e^{s t \cos \phi}\), and \(|d\lambda| = ds\). Since \(\cos \phi < 0\), we have
			\[
			\int_{\varGamma} |e^{\lambda t}| \, |d\lambda| = 2 \int_0^\infty e^{s t \cos \phi} \, ds = 2 \int_0^\infty e^{-s |t \cos \phi|} \, ds = \frac{2}{|t \cos \phi|}.
			\]
	Therefore,  we obtain
			\[
			\left\| \frac{d}{dt} S(t)  \right\|  \leq \frac{C  }{2\pi} \cdot \frac{2}{|t \cos \phi|} =: \frac{C_3 }{ t}.
			\]
				
		\end{proof}

		\begin{lemma}\label{lemma:4.5} For all  $t \geq T_R$ , $\| u_{0,n}\|  \leq R$, and $\theta > 0$, there holds
\begin{equation*}
				\left\| \frac{d}{d\theta} (\mathcal{T}_t \varphi_{u_{0,n}})(\theta) \right\| \leq C_R B_{R}^f {\theta^{ \alpha-1}}.
			\end{equation*}

		\end{lemma}
		\begin{proof} Define $g(\theta) := (\mathcal{T}_t \varphi_{u_{0,n}})(\theta)$. From Lemma \ref{lemma:4.3}, we have
			$$
			g(\theta) = S(t+\theta)u_{0,n} + \int_0^t S(t+\theta - s)f(u(s, u_{0,n})) \, ds.
			$$
			Then,
			$$
			\begin{aligned}
				g'(\theta) &= \frac{d}{d\theta}(\mathcal{T}_t \varphi_{u_{0,n}})(\theta) \\
				&= \frac{d}{d\theta} \left[ S(t+\theta)u_{0,n} \right] + \frac{d}{d\theta} \left[ \int_0^t S(t+\theta - s)f(u(s, u_{0,n})) \, ds \right] \\
				&=: A(\theta) + B(\theta).
			\end{aligned}
			$$
			
 	By Lemma \ref{lemma:4.4}, it yields
			$$
			\left\| \frac{d}{d\sigma} S(\sigma)v \right\|  \leq C_3 \sigma^{  - 1} \| v\| , \quad \sigma > 0.
			$$
			Letting $\sigma = t + \theta$, we have $d\sigma=d\theta$ and
$
			A(\theta) =  \frac{d}{d\sigma} S(\sigma)u_{0,n}$,
			so
			$$ \| A(\theta)\|   =   \left\| \frac{d}{d\sigma} S(\sigma)u_{0,n} \right\| \leq C_3 (t + \theta)^{ - 1} \| u_{0,n}\|  .
			$$
		Since $\|u_{0,n}\| \leq R$, it follows that
			$$
			\|A(\theta)\| \leq C_3 R T_R^{-\alpha } (t + \theta)^{\alpha - 1},~~~t\geq T_R,
			$$
Moreover, for $\theta > 0$, $(t + \theta)^{\alpha - 1} \leq \theta^{\alpha - 1}$ because $\alpha - 1 < 0$ and $t + \theta \geq \theta$. Thus,
			$$
			\| A(\theta)\|   \leq C_3 RT_R^{-\alpha } \theta^{\alpha - 1}.
			$$

On the other hand, we have
			$$
			B(\theta) = \frac{d}{d\theta} \left[ \int_0^t S(t + \theta - s)f(u(s, u_0)) \, ds \right] = \int_0^t \frac{d}{d\theta} \left[ S(t + \theta - s)f(u(s, u_0)) \right] \, ds.
			$$
It yields
			$$\begin{aligned}
				\|B(\theta)\| & \leq \int_0^t \|\frac{d}{d\theta}S(t+\theta-s)f(u(s,u_0))\| ds \\&\leq C_3\int_0^t (t+\theta-s)^{ -1}\|f(u(s,u_0))\| ds .
			\end{aligned}$$

We split the integral interval since \(t \ge T_R\), $[0,t] = [0,\,t-T_R] \cup [t-T_R,\,t].$
Accordingly, $$\begin{aligned}\|B(\theta)\| \le &C_3 \left( \int_0^{t-T_R} (t+\theta-s)^{-1}\|f(u(s))\|\,ds+ \int_{t-T_R}^t (t+\theta-s)^{-1}\|f(u(s))\|\,ds\right)\\ := &C_3(I_1+I_2).\end{aligned} $$

For \(s\in[0,t-T_R]\) we have
$\sigma := t+\theta-s \ge \theta + T_R \ge T_R.$
It yields
\[
\sigma^{-1} \le T_R^{-\alpha}\sigma^{\alpha-1} \le T_R^{-\alpha}\theta^{\alpha-1}.
\]
By Lemma \ref{lemma:4.2}, \(\int_0^t \|f(u(s))\|\,ds \le K_R\) for all \(t\ge 0\), hence
\[
I_1 \le T_R^{-\alpha}\theta^{\alpha-1}\int_0^{t-T_R}\|f(u(s))\|\,ds
\le T_R^{-\alpha}K_R\,\theta^{\alpha-1}.
\]

For initial data satisfying \(\|u_{0,n}\|\le R\), there exists a constant \(B_R^f>0\) such that
\[
\|f(u(s,u_{0,n}))\| \le B_R^f,\qquad \forall\,s\ge0.
\]
Thus
\[
I_2 \le B_R^f \int_{t-T_R}^t (t+\theta-s)^{-1}\,ds
= B_R^f \int_{\theta}^{\theta+T_R} \tau^{-1}\,d\tau
= B_R^f \ln\!\Bigl(1+\frac{T_R}{\theta}\Bigr).
\]
Now we treat the logarithmic term. Consider the function
\[
\varphi(\theta) = \frac{\ln(1+T_R/\theta)}{\theta^{\alpha-1}}
= \theta^{1-\alpha}\ln\!\Bigl(1+\frac{T_R}{\theta}\Bigr),\quad \theta>0.
\]
Since \(\alpha\in(0,1)\),
\[
\lim_{\theta\to0^+}\varphi(\theta)=0,\qquad
\lim_{\theta\to\infty}\varphi(\theta)=0,
\]
and \(\varphi\) is continuous on \((0,\infty)\), therefore \(\varphi\) attains a maximum \(C_{\alpha,T_R}\). Consequently,
\[
\ln\!\Bigl(1+\frac{T_R}{\theta}\Bigr) \le C_{\alpha,T_R}\,\theta^{\alpha-1}.
\]
Thus
\[
I_2 \le C_{\alpha,T_R} B_R^f\,\theta^{\alpha-1}.
\]

Combining these estimates of $I_i$, $i=1,2$, we have
\[
\begin{aligned}
	\|g'(\theta)\|
	&\le \|A(\theta)\| + C_3(I_1+I_2)
 \le C_3\Bigl( R T_R^{-\alpha} + T_R^{-\alpha}K_R + C_{\alpha,T_R} B_R^f \Bigr)\theta^{\alpha-1}.
\end{aligned}
\]
Define
\[
C_R := C_3\Bigl( \frac{R T_R^{-\alpha} + T_R^{-\alpha}K_R}{B_R^f} + C_{\alpha,T_R} \Bigr),
\]
We finally obtain
\[
\Bigl\|\frac{d}{d\theta}(\mathcal{T}_t \varphi_{u_{0,n}})(\theta)\Bigr\|
\le C_R\,B_R^f\,\theta^{\alpha-1},
\qquad \forall\,t\ge T_R,\;\|u_{0,n}\|\le R,\;\theta>0.
\]
\end{proof}

		\subsubsection{ Equi-Lipschitz continuous}
		
		Write $\chi_n(t,\theta):=(\mathcal{T}_t \varphi_{u_{0,n}})(\theta)$, so $\chi_n(0,\theta):=S(\theta)u_{0,n}$. By estimate \eqref{ineq4.4},
		\begin{equation}\label{bounded1-2}
			\|\chi_n(t,\theta)\|  \leq C_R B_{R^*}^f N^{\alpha}+R^{*},\quad t\geq T_{R},~0\leq\theta\leq N.
		\end{equation}
		In addition, by Lemma \ref{lemma:4.5}
		$$
		\left\|\frac{d}{d\theta}\chi_n(t,\theta)\right\| \leq  C_R B_{R}^f \theta^{\alpha-1}, \quad t \geq T_R,$$
		For any \(t \geq T_R\),
if \(0 \leq \theta_1 < \theta_2 \leq N\), we have
		\[
		\|\chi_n(t_n, \theta_2) - \chi_n(t_n, \theta_1)\|  \leq C_R B_R^f |\theta_2^\alpha - \theta_1^\alpha|
		\leq C_R B_R^f (\theta_2  - \theta_1)^\alpha,
		\]
		due to $\xi^\alpha_2-\xi_1^\alpha\leq (\xi_2-\xi_1)^\alpha$ for $0\leq \xi_1<\xi_2$ and $\alpha\in[0,1]$. Therefore, for given \(\varepsilon > 0\), there is a 
		\(\delta>0\) such that 
		whenever \(|\theta_2 - \theta_1| < \delta\) we always have
		\[
		\|\chi_n(t_n, \theta_2) - \chi_n(t_n, \theta_1)\|  < \varepsilon.
		\]
		Thus, for any fixed \(\varepsilon > 0\), the sequence \(\{\chi_n(t_n, \cdot)\}_{n \in \mathbb{N}}\) is uniformly equi-Lipschitz continuous in \(C([0, N], L^p(\mathbb{R}^d))\).

		\subsubsection{Spatial decay}
		\begin{lemma}\label{lemma:4.6}
			For any \(t > 0\), \(v \in L^p(\mathbb{R}^d)\), and \(R > 0\), then there exist constants \(C_4, c_4 > 0\) such that
			\begin{equation}
				\|S(t)v\|_{L^p(|x| > R)} \leq C_4 e^{-c_4R^2 / t} \|v\|  + \|v\|_{L^p(|x| > R/2)}.
			\end{equation}
		\end{lemma}
		\begin{proof}
			First, the heat semigroup \(\mathcal{Q}(\tau)=e^{\tau \Delta}\) has the kernel \(K_\tau(x) = (4\pi\tau)^{-d/2} e^{-|x|^2/(4\tau)}\). For any \(v \in L^p\), decompose \(v\) as \(v = v_1 + v_2\), where \(v_1 = v \chi_{|y| \leq R/2}\) and \(v_2 = v \chi_{|y| > R/2}\). Then
			\[
			e^{\tau \Delta}v = e^{\tau \Delta}v_1 + e^{\tau \Delta}v_2.
			\]
			
			For \(e^{\tau \Delta}v_2\), by the contractivity of the heat semigroup, we have
			\[
			\|e^{\tau \Delta}v_2\|_{L^p(|x| > R)} \leq \|e^{\tau \Delta}v_2\|_{L^p} \leq \|v_2\|_{L^p} = \|v\|_{L^p(|x| > R/2)}.
			\]
			
			For \(e^{\tau \Delta}v_1\), when \(|x| > R\), we have \(|x - y| \geq |x| - |y| > R/2\). Using the Gaussian estimate \(K_\tau(x-y) \leq C \tau^{-d/2} e^{-|x-y|^2/(4\tau)}\), and noting that for \(|x| > R\) we have \(|x - y| \geq |x| - R/2 \geq |x|/2\), it follows that
			\[
			K_\tau(x-y) \leq C \tau^{-d/2} e^{-|x|^2/(16\tau)}.
			\]
			Hence, we have
			\[
			|e^{\tau \Delta}v_1(x)| \leq C \tau^{-d/2} e^{-|x|^2/(16\tau)} \int_{|y| \leq R/2} |v(y)| \, dy \leq \tilde{C}_1 \tau^{-d/2} e^{-|x|^2/(16\tau)} \|v\|_{L^p} (R^d)^{1/p'}.
			\]
			Taking the \(L^p\) norm over \(|x| > R\), we obtain
			\[
			\|e^{\tau \Delta}v_1\|_{L^p(|x| > R)} \leq \tilde{C}_1 \tau^{-d/2} (R^d)^{1/p'} \|v\|  \left( \int_{|x| > R} e^{-p|x|^2/(16\tau)} \, dx \right)^{1/p}.
			\]
			Computing the integral,
			\[
			\int_{|x| > R} e^{-p|x|^2/(16\tau)} \, dx = \omega_d \int_R^\infty r^{d-1} e^{-p r^2/(16\tau)} \, dr = \omega_d \tau^{d/2} \int_{R/\sqrt{\tau}}^\infty j^{d-1} e^{-p j^2/16} \, dj,
			\]
			where \(\omega_d\) is the surface area of the unit sphere and \(r = \sqrt{\tau} j\). For \(j \ge R/\sqrt{\tau}\), we have \(e^{-p j^2/16} \le e^{-p R^2/(32\tau)} e^{-p j^2/32}\). Therefore,
			\[
			\int_{R/\sqrt{\tau}}^\infty j^{d-1} e^{-p j^2/16} \, dj \le e^{-p R^2/(32\tau)} \int_0^\infty j^{d-1} e^{-pj^2/32} \, dj \le C' e^{-p R^2/(32\tau)}.
			\]
			Substituting back, we get
			\begin{align*}
				\|e^{\tau \Delta}v_1\|_{L^p(|x|>R)} \le & \tilde{C}_1 {C'}^{1/p} \tau^{-d/2} \tau^{d/(2p)} R^{d/p'} e^{-R^2/(32\tau)} \|v\| \\ \le & C'' \tau^{-d/(2p')} R^{d/p'} e^{-R^2/(32\tau)} \|v\| ,
			\end{align*}
			where \(C, \tilde{C}_1, C', C''\) are constants, and \(-d/2 + d/(2p) = -d/(2p')\).
			
			For any fixed \(c' < \frac{1}{32}\), consider
			\[
			\tau^{-d/(2p')} e^{-R^2/(32\tau)} = e^{-c' R^2/\tau} \cdot \tau^{-d/(2p')} e^{-(1/32 - c')R^2/\tau}.
			\]
			Let \(s = \tau / R^2\). Then
			\[
			\tau^{-d/(2p')} e^{-(1/32 - c')R^2/\tau} = R^{-d/p'} s^{-d/(2p')} e^{-(1/32 - c')/s}.
			\]
			Since the function \(f(s) = s^{-d/(2p')} e^{-(1/32 - c')/s}\) is bounded on \((0,\infty)\) (tending to 0 as \(s\to 0\) and as \(s\to\infty\)), we set
			\[
			K(c') = \sup_{s>0} s^{-d/(2p')} e^{-(1/32 - c')/s} < \infty.
			\]
			Thus,
			\[
			\tau^{-d/(2p')} e^{-R^2/(32\tau)} \le K(c') R^{-d/p'} e^{-c' R^2/\tau}.
			\]
			Substituting into the estimate for \(v_1\),
			\[
			\|e^{\tau \Delta}v_1\|_{L^p(|x|>R)} \le C'' \cdot K(c') R^{d/p'} R^{-d/p'} e^{-c' R^2/\tau} \|v\| \le C_4 e^{-c' R^2/\tau} \|v\| .
			\]
			
			Combining the contribution from \(v_2\), we obtain
			\[
			\|e^{\tau \Delta}v\|_{L^p(|x|>R)} \le C_4 e^{-c' R^2/\tau} \|v\|  + \|v\|_{L^p(|x|>R/2)}.
			\]
			By Proposition 4.10 in \cite{Gripenberg1990}, the measure \(\mu_t(d\tau)=-d_\tau\omega(t,\tau)\) is a positive measure supported on \([0,t/\kappa]\). Applying this to the operator \(S(t)v = \int_0^{t/\kappa} e^{\tau \Delta}v \,\mu_t(d\tau)\), and noting that \(e^{-c' R^2/\tau} \le e^{-c' \kappa R^2/t}\), we have
			\[
			\|S(t)v\|_{L^p(|x|>R)} \le C_4 e^{-c' \kappa R^2/t} \|v\|  + \|v\|_{L^p(|x|>R/2)}.
			\]
			Setting $c_4 = c' \kappa$ yields the desired estimate. \end{proof}
		
		\begin{lemma}\label{lemma:4.7}  Let $\{u_{0,n}\}_{n\in\mathbb{N}}\subset L^p(\mathbb{R}^d)$ have uniform spatial decay, i.e., for any $\varepsilon>0$, there exists $R_0>0$ such that
			\[
			\sup_n \|u_{0,n}\|_{L^p(|x|>R_0)}<\varepsilon.
			\]
			Assume that the nonlinearity $f$ satisfies ($F_l$) and the dissipativity condition ($F_d$). Then there exists a constant $T_0>0$ such that for all $t_n\le T_0$, and for any $\eta>0$ and $N>0$, there exist $R>0$ and $N_1\in\mathbb{N}$ such that for all $n\ge N_1$ and $\theta\in[0,N]$,
			\[
			\|\chi_n(t_n,\theta)\|_{L^p(|x|>R)}<\eta.
			\]
		\end{lemma}
		\begin{proof}  First, we prove that on any fixed interval $[0,T]$, the solutions $u_n(s)$ have uniform spatial decay. Here we take $T=T_0+N$. Denote $\phi_n(s,R)=\|u_n(s)\|_{L^p(|x|>R)}$. From the integral equation
			\[
			u_n(s)=S(s)u_{0,n}+\int_0^s S(s-\tau)f(u_n(\tau))\,\mathrm{d}\tau,
			\]
			and using Lemma \ref{lemma:4.6}, we obtain
			\[
			\begin{aligned}
				\phi_n(s,R)&\le C_4 e^{-c_4 R^2/s}\|u_{0,n}\| +\|u_{0,n}\|_{L^p(|x|>R/2)}\\
				&\quad +\int_0^s \Bigl(C_4 e^{-c_4R^2/(s-\tau)}\|f(u_n(\tau))\| +\|f(u_n(\tau))\|_{L^p(|x|>R/2)}\Bigr)\mathrm{d}\tau.
			\end{aligned}
			\]
			By uniform boundedness, $\|u_{0,n}\| \le B_R$ and $\|f(u_n(\tau))\| \le B_R^f$. Moreover, from the assumption $f$, there exist a constant $L_R>0$ and a function $\delta(r)\to0$ (as $r\to\infty$) such that
			\[
			\|f(u)\|_{L^p(|x|>r)}\le L_R\|u\|_{L^p(|x|>r)}+\delta(r).
			\]
			Thus, we have
			\begin{equation}\label{A1(R)}
				\phi_n(s,R)\le A_1(R)+L_R\int_0^s \phi_n(\tau,R/2)\,\mathrm{d}\tau,
			\end{equation}
			where
			\[
			A_1(R)=C_4 e^{-c_4R^2/T}B_R+\rho(R/2)+C_4 B_R^f \int_0^T e^{-c_4R^2/(T-\tau)}\mathrm{d}\tau+T\delta(R/2),
			\]
			with $\rho(r)=\sup_n\|u_{0,n}\|_{L^p(|x|>r)}$, and $\rho(r)\to0$. Since $e^{-c_4R^2/(T-\tau)}\le e^{-c_4R^2/T}$ and the integral is bounded, it is easy to see that $A_1(R)\to0$ as $R\to\infty$.
			
			Iterating (\ref{A1(R)}) by $k$-times yields
			\[
			\phi_n(s,R)\le \sum_{j=0}^{k-1}\frac{(L_R s)^j}{j!}A_1(R/2^j)+\frac{(L_R s)^k}{k!}\sup_{\tau\in[0,T]}\phi_n(\tau,R/2^k).
			\]
			Since $\sup_{\tau}\phi_n(\tau,R/2^k)\le B_R$ and $s\le T$, for sufficiently large $k$, the second term is less than $\eta/4$ (independent of $R$). Fix such a $k$, and then choose $R$ sufficiently large so that $\sum_{j=0}^{k-1}\frac{(L_RT)^j}{j!}A_1(R/2^j)<\eta/4$. Then for all $n$ and $s\in[0,T]$, we have $\phi_n(s,R)<\eta/2$. This proves the uniform spatial decay of $u_n(s)$ on $[0,T]$.
			
			Now return to $\chi_n(t_n,\theta)$. Since $t_n\le T_0$ and $\theta\le N$, we have $t_n+\theta\in[0,T]$. Applying Lemma \ref{lemma:4.6} to the first part,
			\[
			\|S(t_n+\theta)u_{0,n}\|_{L^p(|x|>R)}\le C_4 e^{-c_4R^2/(t_n+\theta)}\|u_{0,n}\|_{L^p}+\|u_{0,n}\|_{L^p(|x|>R/2)}.
			\]
			Because $t_n+\theta\le T$, we have $e^{-c_4R^2/(t_n+\theta)}\le e^{-c_4R^2/T}$, and $\|u_{0,n}\|_{L^p}$ is bounded, so we can take $R$ sufficiently large so that the first term is less than $\eta/4$. Moreover, by the uniform decay of the initial data, we can choose $R$ large enough so that the second term is also less than $\eta/4$.

			For the integral term, for each $s\in[0,t_n]$, we also have
			\[
			\|S(t_n+\theta-s)f(u_n(s))\|_{L^p(|x|>R)}\le C_4 e^{-c_4R^2/(t_n+\theta-s)} B_R^f + \|f(u_n(s))\|_{L^p(|x|>R/2)}.
			\]
 	Since the denominator $t_n+\theta-s$ satisfies  $\theta \;\le\; t_n+\theta-s \;\le\; t_n+\theta \;\le\; T_0 + N = T$, we obtain
			$e^{-c_4R^2/(t_n+\theta-s)} \le e^{-c_4R^2/T}$. Hence, after integration, the first term is bounded by $t_n C_4 B_R^f e^{-c_4R^2/T}\le T_0 C_4 B_R^f e^{-c_4R^2/T}$. Choosing $R$ large makes this less than $\eta/4$.
			
			For the second term $\|f(u_n(s))\|_{L^p(|x|>R/2)}$, due to the uniform spatial decay of $u_n(s)$ and the properties of $f$, there exists $R$ such that for all $s\in[0,t_n]$, $\|f(u_n(s))\|_{L^p(|x|>R/2)}<\eta/(4T_0)$. Consequently,  this term is less than $\eta/4$ after taking integration.
			
			Combining the estimates above, by taking $R$ sufficiently large so that all four terms are less than $\eta/4$, we obtain
			\[
			\|\chi_n(t_n,\theta)\|_{L^p(|x|>R)}<\eta.
			\]
			Since all estimates are uniform in $n$ and $\theta$, the lemma is proved. \end{proof}

		\subsubsection{Finite-dimensional approximation}
		\begin{lemma}\label{lemma:4.8}
			There exists a constant \(C_5 > 0\)  such that for any \(t > 0\),
			\begin{equation*}
				\|\nabla S(t)\| \le C_5 \bigl[(1 * m)(t)\bigr]^{-1/2}.
			\end{equation*}
		\end{lemma}
		\begin{proof}
From the gradient estimate of the heat semigroup \(\|\nabla e^{\tau \Delta} v\|  \le C_p \tau^{-1/2} \|v\| \) for \(1 < p < \infty\) for $v\in L^p(\mathbb R^d)$, we obtain
			\[
			\|\nabla S(t)v\|  \le C_p \|v\|  \int_0^\infty \tau^{-1/2} \mu_t(\mathrm d\tau) =: C_p \|v\|  I(t).
			\]

	 From Lemma A.2 of  \cite{Huang2025}, we know that
		\[
		I(t) \leq \Gamma\!\left(\frac12\right)\bigl[(1 * m)(t)\bigr]^{-1/2}.
		\]
	Therefore, we have
		\(
		\|\nabla S(t)\|  \leq C_p\sqrt{\pi}\,[(1*m)(t)]^{-1/2},
		\) which implies the desired estimate.
		\end{proof}
			
		\begin{lemma}\label{lemma:4.9}
			For any fixed \(\theta>0\), there exists a constant \(C_\theta>0\) such that for all \(n\) and \(t_n\ge T_R\),
			\[
			\|\nabla \chi_n(t_n,\theta)\|  \le C_\theta.
			\]
		\end{lemma}
 	
		\begin{proof}
Since \(m\) is completely positive, it means that \(m\ge0\), hence \((1*m)(\tau) \) is non-decreasing with respect to \(\tau\). Thus for any \(\tau\ge\tau_0>0\), it yields \((1*m)(\tau)\ge(1*m)(\tau_0)\), i.e., \([(1*m)(\tau)]^{-1/2}\le[(1*m)(\tau_0)]^{-1/2}\), and consequently from Lemma \ref{lemma:4.8}, we get
	\begin{equation}\label{tau}
		\|\nabla S(\tau)\| \le C_5[(1*m)(\tau_0)]^{-1/2},\qquad\forall\,\tau\ge\tau_0.
\end{equation}
Given that \(\|u_{0,n}\| \le B_R\), taking \(\tau_0=\theta\) yields
		\[
		\begin{aligned}
			\|\nabla S(t_n+\theta)u_{0,n}\|
			 \le \|\nabla S(t_n+\theta)\| \|u_{0,n}\|
			 \le C_5 B_R [(1*m)(\theta)]^{-1/2} =: C_1(\theta).
		\end{aligned}
		\]

		Set
		\[
		I:=\int_0^{t_n}\nabla S(t_n+\theta-s)f(u(s))\,ds,
		\]
		where \(u(s)=u(s,u_{0,n})\) is the solution of equation (1.1).
		Let \(a_n=\max\{0,\,t_n-1\}\) and split the integral as
		\[
		I_1:=\int_{a_n}^{t_n}\nabla S(t_n+\theta-s)f(u(s))\,\mathrm{d}s,\qquad
		I_2:=\int_0^{a_n}\nabla S(t_n+\theta-s)f(u(s))\,\mathrm{d}s.
		\]
		When \(t_n\le 1\), \(a_n=0\), so \(I_2=0\) and \(I_1=I\); when \(t_n>1\), \(a_n=t_n-1\), which reduces to the intervals \([t_n-1,t_n]\) and \([0,t_n-1]\). This decomposition is valid for all \(t_n\ge0\).
		
	Taking \(r=t_n-s\), for \(s\in[a_n,t_n]\) we have \(r\in[0,\,t_n-a_n]\subset[0,1]\), and
		\[
		I_1=\int_0^{t_n-a_n}\nabla S(r+\theta)f(u(t_n-r))\,\mathrm{d}r.
		\]
		Let \(\tau_0=\theta\) in \eqref{tau}, we obtain
		\[
		\|\nabla S(r+\theta)\| \le C_5[(1*m)(\theta)]^{-1/2}.
		\]
		Since there exists a constant \(B_R^f\) such that \(\|f(u(t))\| \le B_R^f\) and the integration length satisfies \(t_n-a_n\le1\), we obtain
		\[
		\begin{aligned}
			\|I_1\|
			 \le C_5 B_R^f\int_0^{t_n-a_n}[(1*m)(\theta)]^{-1/2}\,dr
			 \le C_5 B_R^f[(1*m)(\theta)]^{-1/2} =: C_2(\theta).
		\end{aligned}
		\]
		
		If \(t_n>1\), then \(a_n=t_n-1\). Let \(r=t_n-s\); now \(s\in[0,t_n-1]\) corresponding to \(r\in[1,t_n]\), it follows that
		\[
		I_2=\int_1^{t_n}\nabla S(r+\theta)f(u(t_n-r))\,\mathrm{d}r.
		\]
		For \(r\ge1\) we have \(r+\theta\ge1+\theta\); taking \(\tau_0=1+\theta\) in \eqref{tau} gives
		\[
		\|\nabla S(r+\theta)\| \le C_5[(1*m)(1+\theta)]^{-1/2},
		\]
		which is independent of \(r\). Factoring out the constant and changing the variable back to \(s=t_n-r\), by Lemma \ref{lemma:4.2}, we have
		\[
		\|I_2\| \le C_5[(1*m)(1+\theta)]^{-1/2}\int_0^{t_n-1}\|f(u(s))\|ds \le C_5 K_R [(1*m)(1+\theta)]^{-1/2}=: C_3(\theta).
		\]

		If \(t_n\le1\), then \(I_2=0\) and we simply set \(C_3(\theta):=0\), the estimate still holds.
	
 Now, from the definition of \(\chi_n\),
		\[
		\chi_n(t,\theta)=S(t+\theta)u_{0,n}+\int_0^t S(t+\theta-s)f(u(s))\,\mathrm{d}s,
		\]
Taking the spatial gradient derivative $\nabla$ into above identity, regardless of whether \(t_n\le1\) or \(t_n>1\),  we always have
		\[
		\begin{aligned}
			\|\nabla\chi_n(t_n,\theta)\|
			&\le C_1(\theta)+C_2(\theta)+C_3(\theta) .
		\end{aligned}
		\]
Since \(m=k+1\) and \(k\) is nonnegative, we have \(m\ge1\); hence \((1*m)(\theta)\ge\theta>0\) and \((1*m)(1+\theta)\ge1+\theta>0\), so all constants $C_i(\theta)$, $i=1,2,3$, are finite and positive, which are independent of \(n\) and of the particular \(t_n\ge T_R\). \end{proof}

		Choose a cutoff function \(\xi_R \in C_c^\infty(\mathbb{R}^d)\) such that \(\xi_R(x) = 1\) for \(|x| \leq R\), \(\xi_R(x) = 0\) for \(|x| \geq 2R\), and \(0 \leq \xi_R \leq 1\). Consider the bounded domain \(Q_{2R} = \{x \in \mathbb{R}^d : |x| < 2R\}\). By Lemma \ref{lemma:4.8}, the set \(\{\xi_R \chi_n(t_n, \theta) : n \geq N_1, \theta \in [0, N]\}\) is bounded in \(W^{1,p}(Q_{2R})\) since the gradient is bounded in \(L^p\) and multiplication by a smooth cutoff preserves regularity. Due to the compact embedding \(W^{1,p}(Q_{2R}) \hookrightarrow L^p(Q_{2R})\), there exists a finite-rank projection operator \(P_i : L^p(Q_{2R}) \to L^p(Q_{2R})\) such that for any \(\varepsilon > 0\), there exist \(i \in \mathbb{N}\) and \(N_2 \geq N_1\) satisfying
		\[
		\|(I - P_i)(\xi_R \chi_n(t_n, \theta))\|_{L^p(Q_{2R})} < \eta, \quad \forall n \geq N_2, \theta \in [0, N].
		\]
		
		\subsubsection{Estimate on the Kuratowski noncompactness of measure}
		
		Let \(B_M = \{\chi_n(t_n, \cdot) : n \geq M\} \subset \mathcal{C}([0, N], L^p(\mathbb{R}^d))\). We show that when \(M\) is sufficiently large, \(\kappa(B_M)\) can be made arbitrarily small, where \(\kappa\) denotes the Kuratowski noncompactness of measure in \(\mathcal{C}([0, N], L^p(\mathbb{R}^d))\).
		
		From Lemma \ref{lemma:4.7}, for any \(\eta > 0\), there exist \(R\) and \(N_1\) such that
		\[
		\sup_{\theta \in [0, N]} \|\chi_n(t_n, \theta)\|_{L^p(|x| > R)} < \eta, \quad \forall n \geq N_1.
		\]
		Hence, on the region \(Q_R^c = \{x : |x| > R\}\), we have
		\[
		\kappa_{L^p(Q_R^c)}\left(\{\chi_n(t_n, \theta) : n \geq N_1\}\right) \leq 2\eta.
		\]
		
		On the region \(Q_{2R}\), consider the functions \(\xi_R \chi_n(t_n, \cdot)\). By the statements after Lemma \ref{lemma:4.9}, there exist a finite-dimensional projection \(P_i\) and \(N_2 \geq N_1\) such that
		\[
		\|(I - P_i)(\xi_R \chi_n(t_n, \theta))\|_{L^p(Q_{2R})} < \eta, \quad \forall n \geq N_2, \theta \in [0, N].
		\]
		Therefore, in \(L^p(Q_{2R})\),
		\[
		\kappa\left((I - P_i)\{\xi_R \chi_n(t_n, \theta) : n \geq N_2\}\right) \leq 2\eta.
		\]
		Meanwhile, since \(P_i\) has finite rank, the set \(P_i\{\xi_R \chi_n(t_n, \theta) : n \geq N_2\}\) is uniformly bounded in a finite-dimensional space (due to \(L^p\) boundedness) and hence is precompact, i.e.,
		\[
		\kappa\left(P_i\{\xi_R \chi_n(t_n, \theta) : n \geq N_2\}\right) = 0.
		\]
		Consequently,
		\[
		\kappa_{L^p(Q_{2R})}\left(\{\xi_R \chi_n(t_n, \theta) : n \geq N_2\}\right) \leq 2\eta.
		\]
		Since \(\xi_R \equiv 1\) on \(Q_R\), we obtain
		\[
		\kappa_{L^p(Q_R)}\left(\{\chi_n(t_n, \theta) : n \geq N_2\}\right) \leq 2\eta.
		\]
		
		Combining the estimates on the regions \(Q_R\) and \(Q_R^c\), and using the equicontinuity to lift the pointwise estimates to the space of continuous functions, we obtain in \(\mathcal{C}([0, N], L^p(\mathbb{R}^d))\),
		\[
		\kappa\left(B_{N_2}\right) \leq 4\eta.
		\]
		Since \(\eta > 0\) is arbitrary, this shows that the set \(B_{N_2}\) is totally bounded and hence relatively compact.

		\subsubsection{ Applying Ascoli-Arzel\`a's theorem}
		
		This implies that the infinite-dimensional Ascoli-Arzel\`a theorem can be applied to every interval of the form $[0, N]$, i.e., in the space $\mathcal{C}([0, N], L^p(\mathbb{R}^d))$ of continuous functions $\chi_n: [0, N] \to L^p(\mathbb{R}^d)$. Consequently, there exists a (sub)sequence $t_n \to \infty$ and a function $\chi^* \in \mathcal{C}([0, N], L^p(\mathbb{R}^d))$ such that, for each $N \in \mathbb{N}$, uniformly for $\theta \in [0, N]$, we have
		$$\chi_n(\theta) := \chi_n(t_n, \theta) \to \chi^*(\theta), \quad t_n \to \infty.$$
		

		\subsubsection{  Continuity of $\chi^*(\theta)$ at $\theta=0$}
		
		It follows from Lemma \ref{lemma:4.2}  and the dissipativity condition that
		$$\|u(t+\theta,u_{0,n})-u(t,u_{0,n}) \| \leq C_R B_{R}^f\theta^{\alpha} .$$
		for all $t\geq 0$ and $\theta \geq 0$. Let $t_n\geq T_R$. Then
		$$\|u(t_n+\theta,u_{0,n})-u(t_n,u_{0,n}) \| \leq C_R B_{R^*}^f\theta^{\alpha} ,$$
		for all $\theta \geq 0$.
		
		For each $\theta \ge 0$, there is a convergent subsequence (of the subsequence used above to obtain $\chi^*$) such that limits exist and satisfy
		$$\left\| u^{\ast}(\theta)-u^{\ast} \right\|  \leq  C_R B_{R^*}^f\theta^{\alpha},$$
		It is also clear  that
		$$
		\left\| u^{\ast}(\theta)-\chi^{\ast}(\theta) \right\|  \leq  C_R B_{R^*}^f\theta^{\alpha}, \quad \theta \geq 0.$$
		Thus
		$$\begin{aligned} \left\| \chi^{\ast}(\theta)-u^{\ast} \right\|   \leq \left\| \chi^{\ast}(\theta)-u^{\ast}(\theta) \right\|  + \left\| u^{\ast}-u^{\ast}(\theta) \right\| \leq  2C_R B_{R^*}^f\theta^{\alpha} \to 0, \quad \text{as} \; \theta \to 0. \end{aligned}$$
	Consequently, $\chi^{\ast} \in \mathcal{C}_{\alpha}([0,\infty),L^p(\mathbb{{R}}^{d}))$.
 Thus the operator $(\mathcal{T}_t \varphi_{u_{0,n}})(\cdot)$ is asymptotically compact on bounded intervals, i.e., for every sequence  $t_n \rightarrow \infty$ and 	$\|u_{0,n}\| \leq R$  and there is a subsequence $t_n \rightarrow \infty$ such that $\chi_n(t_n, \cdot) = (\mathcal{T}_{t_n} \varphi_{u_{0,n}})(\cdot) \rightarrow \chi^*(\cdot) \in \mathcal{C}_\alpha([0, \infty), L^p(\mathbb{R}^{d}))$.

		\subsubsection{Estimates in the weighted norm}
		
		In terms of the weighted norm $\|\cdot\|_\alpha$ on $\mathcal{C}_{\alpha}$, the bound (\ref{bounded1-2}) becomes
		$$\begin{aligned}
			\|\chi_n(t, \cdot)\|_{\alpha} & = \|\chi_{n}(t, 0)\|  + \sum_{N=1}^{\infty} \frac{1}{2^{N} N^{\alpha}}\|\chi_{n}(t, \theta)\|_{N}, \\
			& \leq R^{*} + \sum_{N=1}^{\infty} \frac{1}{2^{N} N^{\alpha}}\left(C_R B_{R^*}^f N^{\alpha}+R^{*}\right),
		\end{aligned}$$
		for all $t \geq T_{R}$. Hence
		$$\|\chi_{n}(t, \cdot)\|_{\alpha} \leq R^{*}\left(1 + \sum_{N=1}^{\infty} \frac{1}{2^{N} N^{\alpha}}\right) + C_R B_{R^*}^f  \sum_{N=1}^{\infty} \frac{1}{2^{N}} \leq 2 R^{*} + C_R B_{R^*}^f  =: \widehat{R}^{*},$$
		for all $t \geq T_{R}$.

		\subsubsection{Absorbing set and  attractor}
		
		We will now establish the existence of an absorbing set and a global generalized attractor in the space $\mathcal{C}_{\alpha}$, on which $(\mathcal{T}_t\varphi_{u_{0}})(\cdot)$, $t \geq 0$ forms a semigroup.
		
		Define the closed bounded subset $\mathfrak{B}^{*}$ of $\mathcal{C}_{\alpha}$ by
		$$\mathfrak{B}^{*} := \left\{ \chi \in \mathfrak{C}_{\alpha} : \|\chi\|_{\alpha} \leq 2R^{*} + C_R B_{R^*}^f =: \widehat{R}^{*} \right\}.$$
		This provides an absorbing set for the semigroup $(\mathcal{T}_t\varphi_{u_{0}})(\cdot)$ in $\mathcal{C}_{\alpha}$, i.e., for $t \geq T_{R}$, it absorbs the bounded set of initial data satisfying $\|\varphi_{u_{0}}\|_{\alpha} \leq 2\|u_{0}\| \leq 2R$. Moreover, the semigroup is asymptotically compact.
		
		Consequently, the set $\mathfrak{A}$ given in Theorem \ref{thm:4.3} is a nonempty, closed, and bounded subset of $\mathfrak{B}^{*}$, and it attracts the semigroup $\mathcal{T}_t(\cdot)$ for all initial functions $\varphi_{u_{0}}$.
It remains to prove that the set $\mathfrak{A}$ is invariant under the semigroup $\mathcal{T}_{t}(\cdot)$. First, let $g \in \mathfrak{A}$. Then there exists a bounded sequence $\{u_{0,n}\}_{n \in \mathbb{N}}$ and $t_{n} \to \infty$ such that $\mathcal{T}_{t_{n}}\varphi_{u_{0,n}} \to g$. Let $\tau > 0$ be arbitrary. Then, using the semigroup property and continuity, we have
		$$\mathcal{T}_{\tau + t_{n}}\varphi_{u_{0,n}} = \mathcal{T}_{\tau}(\mathcal{T}_{t_{n}}\varphi_{u_{0,n}}) \to \mathcal{T}_{\tau}g,$$
		which implies $\mathcal{T}_{\tau}\mathfrak{A} \subset \mathfrak{A}$. Alternatively, let $t_{n} = \tau + s_{n}$. Then
		$$\mathcal{T}_{t_{n}}\varphi_{u_{0,n}} = \mathcal{T}_{\tau+s_{n}}\varphi_{u_{0,n}} = \mathcal{T}_{\tau}(\mathcal{T}_{s_{n}}\varphi_{u_{0,n}}) \to \mathcal{T}_{\tau}g.$$
		By asymptotic compactness, $\mathcal{T}_{s_{n}}\varphi_{u_{0,n}} \to f$ (or a subsequence thereof). Hence, by continuity, $\mathcal{T}_{\tau}(\mathcal{T}_{s_{n}}\varphi_{u_{0,n}}) \to \mathcal{T}_{\tau}f$. But $\mathcal{T}_{t_{n}}\varphi_{u_{0,n}} \to g$, so $\mathcal{T}_{\tau}f = g$. Therefore, $\mathfrak{A} \subset \mathcal{T}_{\tau}\mathfrak{A}$, i.e., $\mathfrak{A} = \mathcal{T}_{\tau}\mathfrak{A}$.
		
		This completes the proof of Theorem \ref{thm:4.3}.
		
\vskip0.5cm
\noindent
{\bf Funding declaration:} This paper was supported by   the Guangxi Natural Science Foundation (2026GXNSFAA00640023).\\[2mm]
{\bf Author contributions:} All authors read and approved the final manuscript.\\[2mm]
{\bf Conflict of interest:} The authors declare that they have no competing interest.\\[2mm]
{\bf Data availability statement:} Data sharing is not applicable to this article as no datasets were generated or analyzed during this study.

	\end{document}